\documentclass{article}

\usepackage[tbtags]{amsmath}    
\usepackage{amsfonts,amssymb,amsthm,euscript,makeidx,pifont,boxedminipage,fullpage}
\usepackage{bbm}
\usepackage{array}
\usepackage{enumerate}
\usepackage{braket}
\usepackage{stmaryrd}
\usepackage{subfigure}

\usepackage{fancybox}

\usepackage{ifpdf}
\ifpdf
   \usepackage[pdftex,bookmarks=true,bookmarksnumbered,%
    plainpages=false,hypertexnames=false,pdfpagelabels,%
    colorlinks=true,linkcolor=black,urlcolor=blue,citecolor=red,%
    pdfstartview=FitH]{hyperref}
   \usepackage[pdftex]{graphicx,color}
\else
    \usepackage[dvips,bookmarks=true,bookmarksnumbered,plainpages=false,%
     colorlinks=true,linkcolor=black,urlcolor=blue,citecolor=red,%
     hypertexnames=false,pdfstartview=FitH,breaklinks=true]{hyperref}
    \usepackage{pstricks,pst-node,pst-text,pst-3d}
    \usepackage[dvips]{graphicx}
\fi
\usepackage{color}
\usepackage[active]{srcltx}%
\usepackage{subfigure}
\usepackage{multirow}

\makeatletter

\usepackage{float,MnSymbol}

\usepackage[latin1]{inputenc} 

\usepackage{amsfonts,euscript,makeidx,pifont,boxedminipage,fullpage}
\usepackage{bbm}
\usepackage{array}
\usepackage{enumerate}
\usepackage{braket}
\usepackage{stmaryrd}
\usepackage{subfigure}
\usepackage{centernot}
\usepackage[normalem]{ulem}
\usepackage{soul}

\usepackage{fancybox}

\usepackage{ifpdf}
\ifpdf
   \usepackage[pdftex,bookmarks=true,bookmarksnumbered,%
    plainpages=false,hypertexnames=false,pdfpagelabels,%
    colorlinks=true,linkcolor=black,urlcolor=blue,citecolor=red,%
    pdfstartview=FitH]{hyperref}
   \usepackage[pdftex]{graphicx,color}
\else
    \usepackage[dvips,bookmarks=true,bookmarksnumbered,plainpages=false,%
     colorlinks=true,linkcolor=black,urlcolor=blue,citecolor=red,%
     hypertexnames=false,pdfstartview=FitH,breaklinks=true]{hyperref}
    \usepackage{pstricks,pst-node,pst-text,pst-3d}
    \usepackage[dvips]{graphicx}
\fi
\usepackage{color}
\usepackage[active]{srcltx}%
\usepackage{subfigure}
\usepackage{multirow}

\makeatletter

\newtheorem{theorem}{Theorem}[section]
\newtheorem{lemma}[theorem]{Lemma}

\newcommand{\interior}[1]{%
  {\kern0pt#1}^{\mathrm{o}}%
}

\newcommand{\norm}[1]{\left\| #1 \right\|}

 \newcommand{\lsup}[1]{\underset{#1\to\infty}{\overline{\lim}}}
\newcommand{\linf}[1]{\underset{#1\to\infty}{\underline{\lim}}}

\makeatother
\begin{document}

\title{Large Deviations of Non-Stochastic Interacting Particles on Random Graphs}

\author{James MacLaurin\footnote{New Jersey Institute of Technology. james.n.maclaurin@njit.edu}}
\maketitle
\abstract{
This paper concerns the large deviations of a system of interacting particles on a random graph. There is no time-varying stochasticity, and the only sources of disorder are the random graph connections, and the initial condition. The average number of afferent edges on any particular vertex must diverge to infinity as $N\to \infty$, but can do so at an arbitrarily slow rate. These results are thus accurate for both sparse and dense random graphs. A particular application to sparse Erdos-Renyi graphs is provided. The theorem is proved by pushing forward a Large Deviation Principle for a `nested empirical measure' generated by the initial conditions to the dynamics. The nested empirical measure can be thought of as the density of the density of edge connections: the associated weak topology is more coarse than the topology generated by the graph cut norm, and thus there is a broader range of application.
}
\section{Interacting Particle Systems on Heterogeneous Graphs}

Understanding the behavior of interacting particle systems on heterogeneous graphs has been a topic of increasing interest over the last five years. This work has a diverse range of applications, including neuroscience \cite{Bressloff2012,Huang2017}, mean-field games \cite{Carmona2018,Bayraktar2020}, Josephson arrays \cite{Phillips1993} and machine learning \cite{Montanari2019}. Of the papers on the Large Deviations of interacting particle systems, the majority concern systems with stochasticity \cite{Tanaka1982,Dawson1987,DaiPra1996,Budhiraja2012,Delattre2016,Coppini2019,Oliveira2019}. Even if there is disorder in the graph connections (such as in \cite{Delattre2016,Coppini2019,Oliveira2019,Oliveira2020,Lucon2020}),  in the large size limit the noise from the stochasticity dominates any disorder resulting from the connection topology. By contrast, in this paper we prove a Large Deviations result for networks without any stochasticity, so that the only source of randomness is the random graph topology and possibly the initial condition.

The theory of Large Deviations concerns the probability of rare events \cite{Donsker1975a,Dupuis1997,Dembo1998,Freidlin2012}. In neuroscience, for instance, it is known that networks of neurons can support transient phenomena such as UP / DOWN transitions \cite{Pena2018}, neural avalanches and waves \cite{Roberts2019}. It is disputed the extent to which the seeming disorder and chaoticity of much neural activity is due to white noise fluctuations (arising from synaptic transmission failure, for instance), or the underlying disordered structure of networks of neurons \cite{Sompolinsky1988,Huang2017}. Thus there is a need for a theory that elucidates how the disordered structure of neural networks can support a diversity of different behaviors \cite{Roberts2019}. This paper provides a step in this direction.

There is a longstanding mathematical literature on the asymptotic behavior of networks of interacting particle systems, dating from the seminal work of Tanaka \cite{Tanaka1982} and Sznitman \cite{Sznitman1989}. In recent years, effort has been directed towards understanding the asymptotic behavior of networks of interacting particles on large random graphs subject to white noise \cite{Delattre2016,maclaurin2016large,Coppini2019,Oliveira2019,Lacker2019,Bhamidi2019,Kuehn2020,Lucon2020,Oliveira2020,Bayraktar2020,Barre2020}. In our earlier work in \cite{maclaurin2016large} we determined the asymptotic behavior (and Large Deviations) of noisy interacting particle systems on large, sparse and recurrent graphs where the total number of connections is of the same order as the system size. Subsequent work by \cite{Coppini2019,Oliveira2019} determined the limiting dynamics and large deviations of noisy interacting particle systems on sparse graphs where the total number of connections is asymptotically much larger than the total number of nodes in the network. Most recently, \cite{Oliveira2020} determine the large size limiting behavior of interacting particle systems on Galton-Watson trees (extremely sparse networks, where the expected total number of edges is a finite multiple of the number of vertices). In the Large Deviations result of \cite{Coppini2019,Oliveira2019} the white noise dominates the graph disorder in the large size limit, so that the Large Deviations asymptotic for the distribution of the empirical measure scales as $O(\exp(-NC))$ for some constant $C$ that is the infimum of the rate function over a set. Thus, to leading order, the system still behaves like an all-to-all homogeneous network in the large size limit.

The method of this paper is to demonstrate that the empirical measure containing the path-dynamics can be written as a function of the empirical measure of initial conditions. The Large Deviations rate function for the dynamical system is then the push-forward of the Large Deviations rate function of the initial condition. To the best of this author's knowledge, this method was first applied to interacting particle systems by Tanaka \cite{Tanaka1982}; see also the recent exposition in \cite{Coghi2018}. This method has also been used to determine the Large Deviations of extremely sparse interacting particle systems in \cite{maclaurin2016large}. For this method to work, the push-forward mapping has to be continuous (or approximately continuous). For this reason, we require a more refined empirical measure that contains the conditional empirical measure of afferent connections on one node. With this richer topology, we are able to adapt classical weak-convergence methods to our context. Another advantage of employing the weak convergence of empirical measures is that we do not require that the typical number of afferent connections to each particle to be of the same asymptotic order throughout the network: indeed the number of afferent connections to some of the particles can be $O(1)$, and the number of afferent connections to other particles can diverge with $n$. 

Many papers concerning the asymptotic convergence of large networks of random graphs \cite{Chatterjee2011,Bordenave2015,Bhattacharya2017,Bayraktar2020} utilize the graphon theory of Lovasz \cite{Lovasz2012}. The topology on the space of all graphs is defined using the cut distance metric first developed by Frieze and Kannan \cite{Frieze1999}. Lucon \cite{Lucon2020} and Bayraktar \textit{et al} \cite{Bayraktar2020}  employ graphon theory to determine the asymptotic behavior of disordered noisy interacting particle systems. Recently \cite{Dupuis2020} use the cut-distance to determine the large deviations of an interacting particle model in the dense regime (the total number of edges scales as $O(N^2)$). While graphon theory can be used to study the dynamics of dense graphs, its application to dynamical systems on sparse random graphs is complicated by the lack of a suitable regularity theory \cite{Chatterjee2016,Cook2020,Gkogkas2020}. It has recently been suggested that dynamical systems on sparse graphs can be studied using `graphops' \cite{Backhausz2018,Gkogkas2020}, that represent the edge structure of the graph as a special type of operator on the space of nodes. Our approach bears some similarities to the graphop method of \cite{Gkogkas2020}, however perhaps the most important difference is that we have an uncountable number of nodes, and we must therefore carefully construct a topology that ensures the compactness of the push-forward operator $\Psi$. More specifically, we represent the random structure of the graph in a non-standard `nested' empirical measure $\hat{\mu}^n_*$ (this contains the distribution of the connectivities at each node). We can then represent the empirical measure of path-solutions as a push-forward of the empirical measure of initial conditions arising from the graph. In this way we are the first to obtain a large deviations principle for non-stochastic interacting particle systems on sparse random graphs.

We scale the strength of interaction by dividing by the total number of afferent edges. If one were to be merely interested in the large size limit of the network (and not the Large Deviations), then one could equivalently divide by the average number of afferent edges (thanks to the Law of Large Numbers, this is equivalent in the large size limit). However once one is in the Large Deviations regime (which by definition concerns rare events), this equivalence is no longer the case. In particular, the probability of a significant number of vertices being either very sparsely connected (much sparser than the average), or relatively densely connected (much more than the average) is not negligible in the Large Deviations regime. For this reason, in this paper we always scale the net effect of one particle on another by the total number of afferent edges: thus the less edges that are afferent on any particular vertex, the greater effect each connection has (other papers such as \cite{Oliveira2019,Lucon2020} also use this scaling). The implication of this is that the large deviations rate function is uniformly upperbounded by a non-infinite constant. This may seem strange on first appearances, because in classical Large Deviations work on the empirical measure for stochastically-interacting particle systems, the rate function is the push-forward of the Relative Entropy, and therefore can be infinite \cite{Coghi2018}. The reason that the rate function of this paper is bounded is that asymptotically sparse networks can always yield any particular connectivity structure (thanks to the scaling of the effective connection strength by the total connectivity); and the probability of an asymptotically sparse connectivity scales according to equation \eqref{eq: epsilon to zero U}.
 
 The structure of this paper is as follows. In Section 2 we define the interacting particle system and outline our three main results: the first Theorem proves the existence and continuity (with respect to a topology $\tilde{\mathcal{T}}$ that is more refined than the weak topology) of the push-forward map $\Psi$, the second Theorem proves the Large Deviation Principle in the case that the initial empirical measure satisfies an LDP with respect to the weak topology, and the third Theorem applies these results to interacting particle systems on sparse Erdos-Renyi graphs. In Section 3 we prove Theorem \ref{Theorem Sparse LDP}, and in Section 4 we prove Theorems \ref{Theorem Psi} and \ref{Theorem 1}.

\subsection{Notation:}

We write $\mathcal{E} = \lbrace 0,1\rbrace$ (the state space for the connections). For any topological space $\mathcal{X}$, let $\mathcal{B}(\mathcal{X})$ denote the set of all Borelian subsets generated from open and closed subsets of $\mathcal{X}$. Let $\mathcal{C}_b(\mathcal{X})$ denote the set of bounded continuous functions $\mathcal{X} \to \mathbb{R}$ and let $\mathcal{P}(\mathcal{X})$ denote the set of all Borelian probability measures.  Unless otherwise indicated, $\mathcal{P}(\mathcal{X})$ is endowed with the topology of weak convergence $\mathcal{T}_w$ \cite{Billingsley1999}, which is generated by open sets of the form
\[
\big\lbrace \mu \in \mathcal{P}(\mathcal{X}) \; : \big| \mathbb{E}^{\mu}[g] - x \big| < \delta \big\rbrace,
\]
for any continuous and bounded $g: \mathcal{X} \to \mathbb{R}$, any $x,\delta \in \mathbb{R}$.  If $\mathcal{X}$ possesses a metric $d$, then let $d_W$ denote the Wasserstein metric on $\mathcal{P}(\mathcal{X})$, i.e.
\[
d_W(\mu,\nu) = \inf_{\eta}\mathbb{E}^{\eta}[d(x,y)],
\]
the infimum being taken over all $\eta \in \mathcal{P}(\mathcal{X}\times\mathcal{X})$ whose marginal law of the first variable $x$ is $\mu$, and marginal law of the second variable $y$ is $\nu$. Also $\mathcal{C}([0,T],\mathcal{X})$ is endowed with the topology generated by the metric
\begin{equation}
D\big( \lbrace x_s \rbrace_{s\in [0,T]} , \lbrace y_s \rbrace_{s\in [0,T]} \big) = \sup_{s\in [0,T]}d(x_s,y_s).
\end{equation}
$\mathbb{R}^d$ and $\mathcal{E}$ are endowed with the standard Euclidean norm $\norm{\cdot}$.The index set of the particles is written $I_n = \lbrace -n,-n+1,\ldots,n-1,n \rbrace$.

\section{Statement of Problem and Main Results}

Let $w^{ij} \in \lbrace 0, 1\rbrace$ specify the strength of connection between nodes $i$ and $j$. For the moment we make little assumptions on $w^{ij}$, although later on we will take $w^{ij}$ to be sampled randomly from a distribution. Define the total connection strength at node $j$ to be $\kappa_n^j$, i.e.
\begin{equation}
\kappa^j_n =\sum_{k=-n}^n w^{jk}.
\end{equation}
The following dynamics becomes singular when $\kappa^j_n = 0$. Thus we assume that no vertex is disconnected from the rest of the graph, i.e. for all $n > 0$,
\begin{equation} \label{eq: kappa j n}
\inf_{j\in I_n}\kappa^j_n > 0.
\end{equation}
The dynamics of the discrete network is assumed to take the form
\begin{align}\label{eq: u dynamics}
\frac{du^k}{dt} &= G(u^k(t))+ \frac{1}{\kappa^k_n}\sum_{j \in  \Xi^n_k} w^{kj}f\big(u^k(t), u^j(t) \big),
\end{align}
where $\Xi^n_k = \lbrace j\in I_n : w^{kj} =1 \rbrace$, and $I_n = \lbrace -n,-n+1,\ldots,n-1,n\rbrace$. Notice that the interaction is scaled by the total input, similarly to (for instance) \cite{Oliveira2019,Lucon2020}. 

Here $G,f: \mathbb{R}^d \to \mathbb{R}^d$ are bounded and uniformly Lipschitz. The initial conditions $\lbrace u^i_{*} \rbrace_{i \geq 1}$ may also depend on $n$ (this is omitted from the notation). It is assumed that there exists a uniform bound
\begin{equation}
\sup_{n\geq 0}\sup_{k\in I_n} \norm{u^k_{*}} \leq C_{ini}.
\end{equation}
It is assumed that
\begin{equation}
u^j_* = u^k_* \text{ if and only if }j=k.
\end{equation}
The empirical measure of afferent inputs on neuron $j$ is written as
\begin{equation}
\hat{\mu}^n_j = \frac{1}{\kappa^n_j}\sum_{k\in \Xi^n_j}\delta_{u^k_*} \in \mathcal{P}\big(\mathbb{R}^d\big).
\end{equation}
The initial empirical measure is
\begin{equation}\label{eq: empirical measure}
\hat{\mu}^n_* = (2n+1)^{-1} \sum_{j=-n}^n \delta_{( u^j_* , \hat{\mu}^n_j)} \in \mathcal{P}\big( \mathbb{R}^d \times \mathcal{P}(\mathbb{R}^d)\big).
\end{equation}
We note that this empirical measure is richer (i.e. carries more information) than is usual for interacting particle systems (such as for instance \cite{Tanaka1982,Sznitman1989}). This is necessary for us to define the push-forward operator $\Psi$ of Theorem \ref{Theorem Psi}. The empirical measure, without the connections, is written
\begin{equation}\label{eq: hat hat mu n}
\hat{\hat{\mu}}^n = (2n+1)^{-1}\sum_{j=-n}^n \delta_{u^j_*}.
\end{equation}
Let 
\begin{equation}
\mathcal{D} = \big\lbrace y\in \mathbb{R}^d \; : \norm{y}  \leq C_{ini}  \big\rbrace .
\end{equation}
Let $\mathcal{P}(\mathcal{D}) \subset \mathcal{P}(\mathbb{R}^d)$ be the set of all measures that are supported on $\mathcal{D}$, i.e. $\mu(y \in \mathcal{D}) = 1$. Observe that $\hat{\mu}^n_* \in \mathcal{P}(\mathcal{D})$. 

Define the empirical measure containing the trajectories upto time $t$,
\begin{equation}\label{eq: empirical measure trajectories}
\hat{\mu}^n_t = \frac{1}{2n+1}\sum_{j \in I_n } \delta_{u^j_{[0,t]} } \in \mathcal{P}\big(\mathcal{C}([0,t],\mathbb{R}^d) \big),
\end{equation}
where $u^j_{[0,t]} := \lbrace u^j_s \rbrace_{0\leq s \leq t}$.  

The significance of our first main result (the following theorem) is that the same mapping $\Psi$ holds for all $n$ and every possible empirical measure (subject to the above assumptions). It will allow us to transfer the convergence / Large Deviations of $\hat{\mu}^n_*$ to the convergence / Large Deviations of $\hat{\mu}^n_t$.
\begin{theorem}\label{Theorem Psi}
(i) For any $T\geq 0$, there exists a measurable map $\Psi: \mathcal{P}(\mathcal{D} \times \mathcal{P}(\mathcal{D})) \mapsto \mathcal{P}( \mathcal{C}([0,T],\mathcal{D}))$ (defined in Lemma \ref{Lemma Psi Limit} below) such that
\begin{equation}
\hat{\mu}^n_T = \Psi \cdot \hat{\mu}^n_*,
\end{equation}
(ii) $\Phi$ is injective.\\
(iii) There exists a topology $\tilde{\mathcal{T}}$ (this is defined at the start of Section \ref{Section Psi Definition}) such that $ \Psi :\big( \mathcal{P}(\mathcal{D} \times \mathcal{P}(\mathcal{D})) , \tilde{\mathcal{T}} \big) \mapsto \big(  \mathcal{P}( \mathcal{C}([0,T],\mathcal{D})), \mathcal{T}_w \big)$ is continuous. The topology $\tilde{\mathcal{T}}$ is Hausdorff and separable, and is a refinement of the weak topology over $\mathcal{P}\big(\mathcal{D} \times \mathcal{P}(\mathcal{D})\big)$. $ \mathcal{P}(\mathcal{D} \times \mathcal{P}(\mathcal{D})) $ is compact with respect to $\tilde{\mathcal{T}}$.
\end{theorem}
The following theorem is the second main result of this paper. It is significant that the space in which the empirical measure lives in - i.e. $ \mathcal{P}(\mathcal{D}\times \mathcal{P}(\mathcal{D})) $ - only needs to be endowed with the topology of weak convergence. 
\begin{theorem}\label{Theorem 1}
Suppose that $\lbrace u^j_* , w^{jk} \rbrace_{j,k \in I_n}$ are (possibly correlated) random variables. Let \newline$\Pi^n \in \mathcal{P}\big( \mathcal{P}(\mathcal{D}\times \mathcal{P}(\mathcal{D})) \big)$ be the probability law of $\hat{\mu}^n_*$. Suppose that $\lbrace \Pi^n \rbrace_{n\in \mathbb{Z}^+}$ satisfy a Large Deviation Principle, where $ \mathcal{P}(\mathcal{D}\times \mathcal{P}(\mathcal{D})) $ is equipped with the weak topology $\mathcal{T}_w$, with good rate function\footnote{A good rate function is lower semicontinuous and has compact level sets \cite{Dembo1998}.} $I: \mathcal{P}\big(\mathcal{D} \times \mathcal{P}(\mathcal{D})\big) \mapsto \mathbb{R}^+$. This means that for open $O \in \mathcal{B}\big(\mathcal{P}(\mathcal{D} \times \mathcal{P}(\mathcal{D}))\big)$ and closed $F \in\mathcal{B}\big(\mathcal{P}(\mathcal{D} \times \mathcal{P}(\mathcal{D}))\big)$,
\begin{align}
\lsup{n}\alpha_n^{-1} \log \Pi^n\big( \hat{\mu}^n_* \in F \big) &\leq -\inf_{\nu \in F}I(\nu)\label{eq: LDP assumed 1 prior} \\
\linf{n}\alpha_n^{-1}\log \Pi^n\big( \hat{\mu}^n_* \in O \big) &\geq -\inf_{\nu \in O}I(\nu),\label{eq: LDP assumed 1} 
\end{align}
where $\alpha_n \to \infty$ as $n\to\infty$. 
%
Then, writing $\tilde{\Pi}^n \in \mathcal{P}\big(\mathcal{P}(\mathcal{C}([0,T],\mathbb{R}^d ))\big)$ to be the probability law of $\hat{\mu}^n_T$, $\lbrace \tilde{\Pi}^n \rbrace_{n\geq 1}$ satisfy a Large Deviation Principle with rate function $H: \mathcal{P}(\mathcal{C}([0,T],\mathbb{R}^d)) \mapsto \mathbb{R}^+$, where $H = I \circ \Psi^{-1}$, i.e.
\begin{align}
\lsup{n}\alpha_n^{-1}\log \tilde{\Pi}^n\big( \hat{\mu}^n_T \in F_T \big) &\leq -\inf_{\nu \in F_T}H(\nu) \\
\linf{n}\alpha_n^{-1}\log \tilde{\Pi}^n\big( \hat{\mu}^n_T \in O_T \big) &\geq -\inf_{\nu \in O_T}H(\nu),
\end{align}
and $F_T,O_T \in \mathcal{B}\big( \mathcal{P}(\mathcal{C}([0,T],\mathbb{R}^d))\big)$ are sets that are (respectively) closed and open with respect to the weak topology.
\end{theorem}

\subsection{Large Deviations of Particle Systems on Sparse Erdos-Renyi Graphs}\label{Subsection LDP Sparse}

We now outline a particular application of the above results to sparse random graphs. In the last ten years there has been an increased emphasis on applying the methods of statistical mechanics to sparse random graphs \cite{Delattre2016,Lacker2019,Bhamidi2019,Dembo2010}.  The first component of $u^j(t)$ is taken to be a static position variable $\theta^j_n$ lying in the ring $(-\pi,\pi]$.  The position variables are evenly spaced, i.e. $\theta^j_n = 2\pi j / (2n+1) \mod \mathbb{S}^1$. The other initial conditions $\lbrace u^{j,p}_{*} \rbrace_{2\leq p\leq d}$ are fixed and non-random, and chosen to be a function of the first position variable, i.e.
\begin{equation}\label{eq: u j star}
u^j_* := U\big(\theta^j_n\big).
\end{equation}
Here $U: \mathbb{S}^1 \to \mathbb{R}^{d}$ is bounded and continuous. We follow the convention that self-connections are always present, i.e. $w^{jj} = 1$ identically. This ensures that \eqref{eq: kappa j n} is satisfied, and therefore the dynamics in \eqref{eq: u dynamics} is non-singular. 

When $i\neq j$, we take $w_{ij} = 1 $ with probability $\rho_n\mathcal{C}(\theta^i_n , \theta^j_n)$, for some prescribed continuous bounded function $\mathcal{C}: \mathbb{S}^1 \times \mathbb{S}^1 \to \mathbb{R}$, such that $\mathcal{C}(\theta,\beta) > 0$. The edges $\lbrace w_{ij} \rbrace_{i,j \in I_n}$ are sampled independently. It is assumed that the following limits exist 
\begin{align}
\lim_{n\to\infty}\rho_n =0  \; \; , \; \;
\lim_{n\to\infty}\lbrace  n\rho_n \rbrace = \infty. \label{eq: rho n assumption}
\end{align}
The rest of the assumptions are as outlined in the previous section. Since $\lim_{n\to\infty}\rho_n = 0$, the network is \textit{sparse}, such that the typical number of edges afferent on each node is much less than the system size. One could easily obtain similar results in the case that $\rho_n = O(1)$.

 Define $\Gamma:  \mathbb{S}^1 \times \mathcal{P}(\mathbb{S}^1)\to \mathbb{R}^d \times \mathcal{P}(\mathbb{R}^d)$ to be the function
\begin{equation}\label{eq: Gamma map}
(\theta,\mu) \to \big(U(\theta), \mu \circ U^{-1}),
\end{equation}
and we recall the definition of $U$ in \eqref{eq: u j star}. The second main result of this paper is the following. 

\begin{theorem}\label{Theorem Sparse LDP}
Suppose that \eqref{eq: u j star} and \eqref{eq: rho n assumption} hold, and that $\lbrace w^{jk} \rbrace_{j,k \in I_n}$ are sampled independently.

Let $F_T,O_T \subset \mathcal{P}\big(\mathcal{C}([0,T],\mathbb{R}^d ) \big)$, with $F_T$ closed and $O_T$ open (with respect to the weak topology $\mathcal{T}_w$). Then
\begin{align}
\lsup{n}\lbrace \rho_n (2n+1)^2 \rbrace^{-1}\log \mathbb{P}\big( \hat{\mu}^n_T \in F_T \big) &\leq -\inf_{\nu \in F_T}H(\nu) \\
\linf{n}\lbrace \rho_n (2n+1)^2 \rbrace^{-1}\log \mathbb{P}\big( \hat{\mu}^n_T \in O_T \big) &\geq -\inf_{\nu \in O_T}H(\nu).
\end{align}
Here $H : \mathcal{P}\big( \mathcal{C}([0,T],\mathbb{R}^d ) \big) \to \mathbb{R}^+$ has compact level sets, and is such that
\begin{align}
H(\nu) &= \inf_{\mu \in \mathcal{P}(\mathbb{S}^1\times\mathcal{P}(\mathbb{S}^1)) }\big\lbrace \tilde{I}(\mu) : \Psi \cdot (\mu \circ \Gamma^{-1})= \nu \big\rbrace \text{ where } \\
\tilde{I}:&\mathcal{P}(\mathbb{S}^1\times\mathcal{P}(\mathbb{S}^1)) \to \mathbb{R}^+ \text{ is such that }\tilde{I}(\mu) = \mathbb{E}^{\mu}[ I_\theta(\alpha)] \text{ where }\label{eq: tilde I rate function}
\end{align}
in the above expectation $\theta \in \mathbb{S}^1$, $\alpha \in \mathcal{P}(\mathbb{S}^1)$ and $\Gamma$ is defined in \eqref{eq: Gamma map}. $I_\theta : \mathcal{P}(\mathbb{S}^1) \to \mathbb{R}^+$ is defined as follows. In the case that $\mu$ does not have a density, 
\begin{equation}
I_\theta(\mu) = \frac{1}{2\pi}\int_{\mathbb{S}^1} \mathcal{C}(\theta,\eta)d\eta . 
\end{equation}
Otherwise, in the case that $d\mu(\theta) = \frac{\zeta(\theta)}{2\pi} d\theta$, define
\begin{align}
I_\theta(\mu) = \frac{1}{2\pi}\int_{\mathbb{S}^1}\mathcal{C}(\theta,\eta) d\eta - \exp\bigg(-\frac{1}{2\pi}\int_{\mathbb{S}^1} \zeta(\eta) \log\bigg(\frac{\zeta(\eta)}{C(\theta,\eta)} \bigg)d\eta  \bigg),
\end{align}
following the convention that $0\log(0)= 0$ in the second integral on the right hand side.
\end{theorem}
This theorem is proved in Section \ref{Section Sparse LDP}.
\section{Proof of Theorem \ref{Theorem Sparse LDP}}\label{Section Sparse LDP}

We assume the model setup of Section \ref{Subsection LDP Sparse}. We need to show that the LDP for the law of the initial condition in the statement of Theorem \ref{Theorem 1} is satisfied: i.e. we need to demonstrate \eqref{eq: LDP assumed 1 prior} and \eqref{eq: LDP assumed 1} hold for the sparse model of Section \ref{Subsection LDP Sparse}. Define 
\[
\breve{\mu}^n_* = \frac{1}{2n+1}\sum_{j\in I_n} \delta_{(\theta^j_n, \breve{\mu}^n_j)}\in  \mathcal{P}\big(\mathbb{S}^1\times \mathcal{P}(\mathbb{S}^1)\big)
\]
where the conditional empirical measures are, for $-n \leq j \leq n$,
\begin{align}
\breve{\mu}^n_{j} &= \frac{1}{|\Xi^n_j|}\sum_{k \in \Xi^n_j} \delta_{\theta^k_n} \in \mathcal{P}( \mathbb{S}^1 ),
\end{align}
where $\Xi^n_j = \lbrace k\in I_n : w^{jk} = 1 \rbrace$. 
Since $\Gamma$ (as defined in \eqref{eq: Gamma map}) is continuous and $\hat{\mu}^n_* = \breve{\mu}^n_* \circ \Gamma^{-1}$, thanks to Varadhan's Contraction Principle \cite[Theorem 4.2.1]{Dembo1998}, it suffices to prove the LDP for the sequence of laws $\lbrace \breve{\Pi}^n \rbrace_{n\geq 1} \subset \mathcal{P}\big(\mathcal{P}(\mathbb{S}^1\times\mathcal{P}(\mathbb{S}^1))\big)$ of $\breve{\mu}^n_*$.  We first assume that the connection probabilities $\mathcal{C}$ are piecewise constant.

\begin{lemma}
Suppose that there exists a partition of $\mathbb{S}^1$ into intervals, i.e. $\bigcup_{i=1}^Q R_i$ , and such that for each $\alpha \in \mathbb{S}^1$,
\begin{equation} \label{eq: connection probabilities piecewise constant}
\mathcal{C}(\theta_1,\alpha) = \mathcal{C}(\theta_2,\alpha)
\end{equation}
if $\theta_1,\theta_2 \in R_i$ for some $R_i$. Thus $\mathcal{C}$ is piecewise constant in its first argument, and continuous in its second argument.

For open $A \subset \mathcal{P}(\mathbb{S}^1\times\mathcal{P}(\mathbb{S}^1))$ and closed $F \subset \mathcal{P}(\mathbb{S}^1\times\mathcal{P}(\mathbb{S}^1))$ 
\begin{align}
\lsup{n}\big\lbrace (2n+1)^2 \rho_n \big\rbrace^{-1} \mathbb{P}\big(  \breve{\mu}^n_* \in F \big) &\leq -\inf_{\nu \in F} \tilde{I}(\nu)\label{eq:LDP 1}\\
\linf{n}\big\lbrace (2n+1)^2 \rho_n \big\rbrace^{-1} \mathbb{P}\big(  \breve{\mu}^n_* \in A \big) &\geq -\inf_{\nu \in A} \tilde{I}(\nu).\label{eq:LDP 2}
 \end{align}
 \end{lemma}
 \begin{proof}
To prove the LDP we are going to employ the Dawson-Gartner theory of projective limits \cite{Dawson1987}, following the exposition in \cite{Dembo1998}.  Let $\mathcal{W}_m \subset  \mathcal{B}(\mathbb{S}^1 \times \mathcal{P}(\mathbb{S}^1))$ be $\mathcal{W}_m = \big\lbrace \mathcal{O}^m_i \big\rbrace_{i=1}^{m^2}$, where
\[
\mathcal{O}^m_i = \big\lbrace (\theta,\mu) \; : \; \theta \in O^m_{j_i} \text{ and }\mu \in \mathcal{Z}^m_{k_i} \big\rbrace,
\]
for some indices $1\leq j_i,k_i \leq m$. 
Let $\mathcal{T}_2$ be the topology on $\mathcal{P}\big(\mathbb{S}^1 \times \mathcal{P}(\mathbb{S}^1)\big)$ that is generated by open sets of the form
\[
\big\lbrace \mu \in \mathcal{P}\big(\mathbb{S}^1 \times \mathcal{P}(\mathbb{S}^1)\big) \; : | \mu( \mathcal{O}^m_i )- x | < \delta \big\rbrace,
\]
for $x,\delta \in \mathbb{R}$, some $m \geq 1$ and $1\leq i \leq m^2$. 

\begin{lemma}
$\mathcal{P}\big(\mathbb{S}^1 \times \mathcal{P}(\mathbb{S}^1)\big)$  is separable with respect to $\mathcal{T}_2$.\\
Also $\mathcal{T}_2$ is a refinement of the weak topology. 
\end{lemma}
\begin{proof}
Suppose that $V: \mathbb{S}^1 \times \mathcal{P}(\mathbb{S}^1) \to \mathbb{R}$ is continuous, and write $V_{max} = \sup_{\theta\in \mathbb{S}^1,\mu \in \mathcal{P}(\mathbb{S}^1)}\lbrace |V(\theta,\mu)|$. Define the set
\begin{equation}
\mathcal{V} = \big\lbrace \nu \in \mathcal{P}(\mathbb{S}^1\times\mathcal{P}(\mathbb{S}^1)) \; : \big| \mathbb{E}^{\nu}[V] - x \big| < \delta \big\rbrace,
\end{equation}
for some $x,\delta \in \mathbb{R}$. We assume that $\mathcal{V}$ is nonempty. We must prove that $\mathcal{V}$ is open with respect to $\mathcal{T}_2$. Fix $\nu \in \mathcal{V}$ such that $\big| \mathbb{E}^{\nu}[V] - x \big| = \tilde{\delta}$ and $\tilde{\delta} < \delta$. Define $\bar{V}^{(m)} , \underline{V}^{(m)}$ to be
\begin{align*}
\bar{V}^{(m)}(\theta,\mu) &= \sum_{i=1}^{m^2} \chi\big\lbrace (\theta,\mu) \in \mathcal{O}_i^m \big\rbrace \sup_{(\beta,\gamma) \in \mathcal{O}^m_i}V(\beta,\gamma) \\
\underline{V}^{(m)}(\theta,\mu) &= \sum_{i=1}^{m^2} \chi\big\lbrace (\theta,\mu) \in \mathcal{O}_i^m \big\rbrace \inf_{(\beta,\gamma) \in \mathcal{O}^m_i}V(\beta,\gamma).
\end{align*}
Since $V$ is defined on a compact domain, it must be uniformly continuous, and therefore there must exist $m$ such that
\begin{equation}
\sup_{\theta \in \mathbb{S}^1,\mu\in \mathcal{P}(\mathbb{S}^1)}\big| \bar{V}^{(m)}(\theta,\mu) - \underline{V}^{(m)}(\theta,\mu) \big| <  \frac{1}{2}( \delta - \tilde{\delta}).
\end{equation}
We then see that the set
\begin{equation}
\mathcal{U} = \big\lbrace  \mu \in \mathcal{P}(\mathbb{S}^1\times\mathcal{P}(\mathbb{S}^1)) \; : \sup_{1\leq i \leq m}\big| \mu(\mathcal{O}^m_i) - \nu(\mathcal{O}^m_i) \big| < \frac{1}{2V_{max}}( \delta - \tilde{\delta}) \big\rbrace
\end{equation}
is such that $\mathcal{U}\subseteq \mathcal{V}$. This yields the lemma.
\end{proof}

$\mathcal{T}_2$ is a projective limit topology (in the notation of \cite[Section 4.6]{Dembo1998}) in the following sense. For $m \geq 1$, the projection $p_{m}: \mathcal{P}(\mathbb{S}^1 \times \mathcal{P}(\mathbb{S}^1)) \to [0,1]^m$ is 
\[
p_{m}\cdot \nu = (\nu(\mathcal{O}^m_1),\ldots,\nu(\mathcal{O}^m_m)).
\]
Clearly $p_{m}: \mathcal{P}(\mathbb{S}^1 \times \mathcal{P}(\mathbb{S}^1)) \to [0,1]^m $ is continuous (when $\mathcal{P}(\mathbb{S}^1 \times \mathcal{P}(\mathbb{S}^1))$ is endowed with the $\mathcal{T}_2$ topology and $[0,1]^m$ with the Euclidean topology). Let $p_{qm} : [0,1]^m \to [0,q]^q$ (for $q \leq m$) be the induced projection i.e. for $1\leq j \leq q$,
\[
p_{qm}(x_1,\ldots,x_m)^j = \sum_{i=1}^m \chi\lbrace \mathcal{O}^m_i \subseteq \mathcal{O}^q_j \rbrace x_i.
\]
The above identity means that 
\begin{equation}
p_{qm}\cdot p_m(\nu) = p_q(\nu).
\end{equation}
It is clear that $p_{qm}$ is continuous.

To prove the theorem, we will use the Dawson-Gartner Theorem for the Large Deviations of projective limit systems \cite{Dawson1987},\cite[Theorem 4.6.1]{Dembo1998}. To satisfy the conditions of this theorem, we first prove that for any partition $P > 0$ and any set $A \in \mathcal{B}([0,1]^P)$ that is either open or closed,
\begin{align}
\lim_{n\to\infty}\big\lbrace (2n+1)^2 \rho_n \big\rbrace^{-1} \mathbb{P}\big(  (\breve{\mu}^n_*(\mathcal{O}^P_1) , \ldots,\breve{\mu}^n_*(\mathcal{O}^P_P) ) \in A \big) &= -\inf_{x \in A}I_{\mathcal{O}^P_1,\ldots,\mathcal{O}^P_P}(x) \text{ where } \\
I_{\mathcal{O}^P_1,\ldots,\mathcal{O}^P_P}(x)  &=\sum_{i=1}^P \mathbb{E}^{\nu}\big[\chi\lbrace \mathcal{O}^P_i \rbrace  \inf_{\mu \in \mathcal{Z}_i}I_\theta(\mu)\big] ,
 \end{align}
 and we have written $\mathcal{O}^P_i = O^P_i \times \mathcal{Z}_i$. Let 
\[
\mathcal{U}_n = \big\lbrace z_n \in  A \; : z_n^i = p_n^i / (2n+1) \text{ for }p_n^i \in \mathbb{Z}^+ \text{ and }|p_n^i| \leq (2n+1) \text{ and }\sum_{i=1}^P z_i = 1 \big\rbrace .
\]
Notice that $|\mathcal{U}_n|$ is polynomial in $n$, which means that 
\[
 \lim_{n\to\infty}\lbrace (2n+1)^2 \rho_n \rbrace^{-1}\log |\mathcal{U}_n| = 0.
\]
We thus find that,
 \begin{align*}
   \lim_{n\to\infty}\lbrace (2n+1)^2 \rho_n \rbrace^{-1}\log \mathbb{P}\big(  (\breve{\mu}^n_*(\mathcal{O}^P_1) ,& \ldots,\breve{\mu}^n_*(\mathcal{O}^P_P) ) \in A \big)  \\
    &=\lim_{n\to\infty}\lbrace (2n+1)^2 \rho_n \rbrace^{-1}\log \sum_{z_n \in \mathcal{U}_n} \mathbb{P}\big(  (\breve{\mu}^n_*(\mathcal{O}^P_1) , \ldots,\breve{\mu}^n_*(\mathcal{O}^P_P) ) =z_n \big)  \\
    &= \lim_{n\to\infty}\sup_{z_n \in \mathcal{U}_n}\lbrace (2n+1)^2 \rho_n \rbrace^{-1}\log \mathbb{P}\big(  (\breve{\mu}^n_*(\mathcal{O}^P_1) , \ldots,\breve{\mu}^n_*(\mathcal{O}^P_P) ) =z_n \big),
 \end{align*}
 as long as the last limit exists. Now fix $z_n \in \mathcal{U}_n$ to be such that the supremum on the right hand side is attained at $z_n$ (this is possible because $\mathcal{U}_n$ is of finite size).

Let $\lbrace \upsilon^n_j \rbrace_{j\in I_n} \subset \mathcal{P}( \mathbb{S}^1)$ be any set of measures such that (i) $\mathbb{P}( \breve{\mu}^n_j = \upsilon^n_j ) \neq 0 $, and (ii),
defining $\upsilon^n = \frac{1}{2n+1}\sum_{j\in I_n}\delta_{\theta^j_n , \upsilon^n_j} \in \mathcal{P}\big(\mathbb{S}^1\times\mathcal{P}(\mathbb{S}^1)\big)$, 
\[
 ( \upsilon^n(\mathcal{O}^P_1) , \ldots,\upsilon^n(\mathcal{O}^P_P) ) = z_n.
\]
Employing a multinomial expansion, we find that
\begin{equation}
 \mathbb{P}\big(  (\breve{\mu}^n_*(\mathcal{O}^P_1) , \ldots,\breve{\mu}^n_*(\mathcal{O}^P_P) ) = z_n \big) = \Upsilon_{n,z_n} \prod_{j \in I_n}\mathbb{P}\big(\breve{\mu}^n_j = \upsilon_j\big)
 \end{equation}
 where $\Upsilon_{n,z_n}$ contains combinatorial factors, and (thanks to Stirling's Approximation) is such that $\Upsilon_n \leq \exp( Cn)$ for a positive constant $C$.  
 We thus find that (recalling that $\rho_n (2n+1) \to \infty$ as $n\to\infty$), and making use of Lemma \ref{Lemma Exact Asymtotic LDP},
 \begin{align}
 \lim_{n\to\infty}\big\lbrace (2n+1)^2 \rho_n \big\rbrace^{-1}\log \mathbb{P}\big(  (\breve{\mu}^n_*(\mathcal{O}^P_1) , \ldots,\breve{\mu}^n_*(\mathcal{O}^P_P) ) =z_n \big) &= - \inf_{z\in A}\sum_{i=1}^P z^i \inf_{\mu \in \mathcal{Z}_i}I_{\tilde{\theta}_i}(\mu) \nonumber \\
 &= -\inf_{z\in A} I_{\mathcal{O}^P_1,\ldots,\mathcal{O}^P_P}(z).
 \end{align}
 Now, applying the Dawson-Gartner Theorem \cite[Theorem 4.6.1]{Dembo1998}, it must be that for any open $A \subset \mathcal{P}\big(\mathbb{S}^1 \times \mathcal{P}(\mathbb{S}^1) \big)$ and any closed $F \subset \mathcal{P}\big(\mathbb{S}^1 \times \mathcal{P}(\mathbb{S}^1) \big)$ (these sets are open / closed with respect to the $\mathcal{T}_2$ topology),
\begin{align}
\lsup{n}\big\lbrace (2n+1)^2 \rho_n \big\rbrace^{-1} \log \mathbb{P}\big(  \breve{\mu}^n_* \in F \big) &\leq -\inf_{\nu \in F} \breve{I}(\nu)\\
\linf{n}\big\lbrace (2n+1)^2 \rho_n \big\rbrace^{-1} \log \mathbb{P}\big(  \breve{\mu}^n_* \in A \big) &\geq -\inf_{\nu \in A} \breve{I}(\nu) \text{ where } \\
\breve{I}(\nu) &= \sup_{P>0}\big\lbrace \breve{I}_{\mathcal{O}^P_1,\ldots,\mathcal{O}^P_P}\big( (\nu(\mathcal{O}^P_1),\ldots,\nu(\mathcal{O}^P_P))\big) \big\rbrace \label{eq: breve I nu}
 \end{align}
 In Lemma \ref{Lemma Rate Function Equivalence} we demonstrate that $ \breve{I}(\nu) = \tilde{I}(\nu)$. We have thus established \eqref{eq:LDP 1} and \eqref{eq:LDP 2}.
 \end{proof}
 
 We now complete the proof of Theorem \ref{Theorem Sparse LDP}.
 \begin{proof}
We earlier assumed in \eqref{eq: connection probabilities piecewise constant} that the connection probabilities are piecewise constant. We now prove the more general case. Let $\mathcal{C}^{(m)}(\theta,\alpha)$ be piecewise constant in its first argument (satisfying \eqref{eq: connection probabilities piecewise constant}), continuous in its second argument, for some $m\geq 1$, and such that
\begin{equation}
\lim_{m\to\infty} \sup_{\theta,\alpha \in \mathbb{S}^1}\big| \mathcal{C}^{(m)}(\alpha,\theta) - \mathcal{C}(\alpha,\theta) \big| = 0.
\end{equation}
It is always possible to make the above approximation because $\mathcal{C}$ is continuous on the compact space $\mathbb{S}^1 \times \mathbb{S}^1$ (and therefore it must be uniformly continuous). Furthermore, since by assumption $\mathcal{C}(\cdot,\cdot) \geq C_{lb} > 0$, it must be that 
\begin{equation}
\lim_{m\to\infty} \sup_{\theta,\alpha \in \mathbb{S}^1}\bigg| \frac {\mathcal{C}(\alpha,\theta) }{\mathcal{C}^{(m)}(\alpha,\theta) }- 1\bigg| = 0.
\end{equation}
It remains to prove that for open $A \subset \mathcal{P}(\mathbb{S}^1\times\mathcal{P}(\mathbb{S}^1))$ and closed $F \subset \mathcal{P}(\mathbb{S}^1\times\mathcal{P}(\mathbb{S}^1))$ 
\begin{align}
\lsup{n}\big\lbrace (2n+1)^2 \rho_n \big\rbrace^{-1} \mathbb{P}\big(  \breve{\mu}^n_* \in F \big) &\leq -\inf_{\nu \in F} \tilde{I}(\nu)\label{eq:LDP 1 0}\\
\linf{n}\big\lbrace (2n+1)^2 \rho_n \big\rbrace^{-1} \mathbb{P}\big(  \breve{\mu}^n_* \in A \big) &\geq -\inf_{\nu \in A} \tilde{I}(\nu).\label{eq:LDP 2 0}
 \end{align}

Let $Q \in \mathcal{P}\big(\mathcal{E}^{(2n+1)^2}\big)$ be the law of the random connections $\lbrace w^{jk} \rbrace_{j,k\in I_n}$ with connection probability $\rho_n\mathcal{C}(\cdot,\cdot)$. Let $Q^{(m)} \in \mathcal{P}\big(\mathcal{E}^{(2n+1)^2}\big)$ be the law of the random connections $\lbrace w^{jk} \rbrace_{j,k\in I_n}$ with connection probability $\rho_n\mathcal{C}^{(m)}(\cdot,\cdot)$. We see that
\begin{equation}
\frac{dQ}{dQ^{(m)}}\big( \lbrace w^{jk} \rbrace \big) =\prod_{j,k\in I_n}\frac{\chi\lbrace w^{jk} = 1\rbrace \rho_n\mathcal{C}(\theta^j_n,\theta^k_n) + \chi\lbrace w^{jk}=0 \rbrace( 1- \rho_n\mathcal{C}(\theta^j_n,\theta^k_n)) }{\chi\lbrace w^{jk} = 1\rbrace \rho_n\mathcal{C}^{(m)}(\theta^j_n,\theta^k_n) + \chi\lbrace w^{jk}=0 \rbrace( 1- \rho_n\mathcal{C}^{(m)}(\theta^j_n,\theta^k_n)) }.
\end{equation}
Define, for $a\geq 0$,
\[
\mathcal{V}^n_a = \big\lbrace \lbrace w^{jk} \rbrace \in \mathcal{E}^{(2n+1)^2} : \sum_{j,k\in I_n} w^{jk} \leq \rho_n(2n+1)^2 a\big\rbrace ,
\]
and
\begin{equation}
\alpha_m = \max\bigg\lbrace \sup_{\theta,\alpha \in \mathbb{S}^1}\bigg| \frac {\mathcal{C}(\alpha,\theta) }{\mathcal{C}^{(m)}(\alpha,\theta) }- 1\bigg| , \sup_{\theta,\alpha \in \mathbb{S}^1}\big| \mathcal{C}^{(m)}(\alpha,\theta) - \mathcal{C}(\alpha,\theta) \big|   \bigg\rbrace .
\end{equation}
Using the inequality
\[
\frac{1-z}{1-x} \leq 1 + x-z + 2x^2,
\]
as long as $x,z \ll 1$ are sufficiently small, it must hold that there is a $K > 0$ such that
\begin{equation}
 \frac{ 1- \rho_n\mathcal{C}(\theta^j_n,\theta^k_n)}{ 1- \rho_n\mathcal{C}^{(m)}(\theta^j_n,\theta^k_n)} \leq 1 + \rho_n \alpha_m + K\rho_n^2.
 \end{equation}
We thus see that for $\lbrace w^{jk} \rbrace \in \mathcal{V}^n_a$, it must hold that
\begin{align}
\frac{dQ^n}{dQ^n_{(m)}}\big( \lbrace w^{jk} \rbrace \big) \leq &(1+\alpha_m)^{\rho_n(2n+1)^2 a}\big(1 + \rho_n \alpha_m + O(\rho_n^2)\big)^{(2n+1)^2(1-\rho_n a)}\nonumber \\
\leq & \exp\big( 2\alpha_m  \rho_n(2n+1)^2 a + K\rho_n^2 (2n+1)^2\big),\label{eq: dQ dQ m}
\end{align}
using the bound $1+y \leq \exp(y)$. Now for any set $U \in \mathcal{B}\big( \mathcal{P}(\mathbb{S}^1 \times \mathcal{P}(\mathbb{S}^1))\big)$,
\begin{multline}
 \lbrace (2n+1)^2 \rho_n \rbrace^{-1}\log Q^n\big(  \breve{\mu}^n_* \in U \big) \leq \nonumber \\
 \lbrace (2n+1)^2 \rho_n \rbrace^{-1}\log \big\lbrace Q^n\big(  \breve{\mu}^n_* \in U \text{ and }\lbrace w^{jk} \rbrace \in \mathcal{V}^n_a \big) +  Q^n\big( \lbrace w^{jk} \rbrace \notin \mathcal{V}^n_a \big) \big\rbrace
 \leq \lbrace (2n+1)^2 \rho_n \rbrace^{-1}\log 2 \nonumber \\ + \max\bigg\lbrace \lbrace (2n+1)^2 \rho_n \rbrace^{-1}\log  Q^n\big(  \breve{\mu}^n_* \in U \text{ and }\lbrace w^{jk} \rbrace \in \mathcal{V}^n_a \big) 
 , \lbrace (2n+1)^2 \rho_n \rbrace^{-1}\log Q^n\big( \lbrace w^{jk} \rbrace \notin \mathcal{V}^n_a \big) \bigg\rbrace .
 \end{multline}
Let $M > \sup_{\theta,\alpha\in\mathbb{S}^1}\mathcal{C}(\theta,\alpha)$. Thanks to Lemma \ref{Lemma Exponential Tightness}, we can find $a > 0$ such that
 \begin{equation}
  \lbrace (2n+1)^2 \rho_n \rbrace^{-1}\log Q^n\big( \lbrace w^{jk} \rbrace \notin \mathcal{V}^n_a \big) \leq -M .
 \end{equation}
 Starting with the upper bound (and taking $U$ to be closed) we thus find that
 \begin{align}
\lsup{n} \lbrace (2n+1)^2 \rho_n \rbrace^{-1}\log Q^n\big(  \breve{\mu}^n_* \in U \big) \leq - \min\big\lbrace M, \lsup{n} \lbrace (2n+1)^2 \rho_n \rbrace^{-1}\log  Q^n\big(  \breve{\mu}^n_* \in U \text{ and }\lbrace w^{jk} \rbrace \in \mathcal{V}^n_a \big) \big\rbrace  .
 \end{align}
 Employing the bound in \eqref{eq: dQ dQ m},
 \begin{multline*}
 \lsup{n} \lbrace (2n+1)^2 \rho_n \rbrace^{-1}\log  Q^n\big(  \breve{\mu}^n_* \in U \text{ and }\lbrace w^{jk} \rbrace \in \mathcal{V}^n_a \big)\\ \leq \lsup{n} \lbrace (2n+1)^2 \rho_n \rbrace^{-1}\log  Q^n_{(m)}\big(  \breve{\mu}^n_* \in U \text{ and }\lbrace w^{jk} \rbrace \in \mathcal{V}^n_a \big) + 2\alpha_m a.
 \end{multline*}
Furthermore
 \begin{align}
  \lsup{n} \lbrace (2n+1)^2 \rho_n \rbrace^{-1}\log  Q^n_{(m)}\big(  \breve{\mu}^n_* \in U \text{ and }\lbrace w^{jk} \rbrace \in \mathcal{V}^n_a \big) &\leq  \lsup{n} \lbrace (2n+1)^2 \rho_n \rbrace^{-1}\log  Q^n_{(m)}\big(  \breve{\mu}^n_* \in U \big)\nonumber \\
  &\leq -\inf_{\nu \in U}\tilde{I}^{(m)}(\nu).
 \end{align}
 where
 \begin{align}
\tilde{I}^{(m)}(\nu) &= \mathbb{E}^{\nu(\theta,\mu)}[ I^{(m)}_\theta(\mu) ] \text{ and }\\
I^{(m)}_\alpha(\mu) &= \frac{1}{2\pi}\int_{\mathbb{S}^1}\mathcal{C}^{(m)}(\alpha,\theta) d\theta - \exp\bigg(-\frac{1}{2\pi}\int_{\mathbb{S}^1} \zeta(\theta) \log\bigg(\frac{\zeta(\theta)}{\mathcal{C}^{(m)}(\alpha,\theta)} \bigg)d\theta  \bigg).
\end{align}
Since $\mathcal{C}^{(m)}$ converges uniformly to $\mathcal{C}$, one easily checks that
\[
\lim_{m\to\infty}\sup_{\alpha \in \mathbb{S}^1}\sup_{\mu  \in \mathcal{P}(\mathbb{S}^1) }\big\lbrace \big| I^{(m)}_\alpha(\mu) - I_\alpha(\mu)\big| \big\rbrace = 0.
\]
We now take $m\to\infty$. Since 
\[
\inf_{\nu \in U}\tilde{I}(\nu) \leq \sup_{\theta,\alpha\in\mathbb{S}^1}\mathcal{C}(\theta,\alpha) < M,
\]
it must be that
\begin{equation}
\lim_{m\to\infty} \min\big\lbrace M, \lsup{n} \lbrace (2n+1)^2 \rho_n \rbrace^{-1}\log  Q^n\big(  \breve{\mu}^n_* \in U \text{ and }\lbrace w^{jk} \rbrace \in \mathcal{V}^n_a \big) \big\rbrace  
\leq  -\inf_{\nu \in U}\tilde{I}(\nu)
\end{equation}
 and therefore \eqref{eq:LDP 1 0} must hold.
 
 The lower bound is prove analogously.

\end{proof}

\begin{lemma}\label{Lemma Rate Function Equivalence}
 \begin{equation}
 \breve{I}(\nu) = \tilde{I}(\nu),
 \end{equation}
 where $\tilde{I}$ is defined in \eqref{eq: tilde I rate function} and $\breve{I}$ is defined in \eqref{eq: breve I nu}.
\end{lemma}
\begin{proof}
Let $P > 0$ be such that
\begin{equation}
\breve{I}_{\mathcal{O}^P_1,\ldots,\mathcal{O}^P_P}\big( (\nu(\mathcal{O}^P_1),\ldots,\nu(\mathcal{O}^P_P))\big) \geq \sup_{Q\geq 1}\big\lbrace \breve{I}_{\mathcal{O}^Q_1,\ldots,\mathcal{O}^Q_Q}\big( (\nu(\mathcal{O}^Q_1),\ldots,\nu(\mathcal{O}^Q_Q))\big) \big\rbrace - \epsilon,
\end{equation}
for some $\epsilon\ll 1$. For any sequence $\lbrace r_P \rbrace_{P\geq 1}$ such that  
\begin{equation}
\mathcal{O}_{r_{P+1}}^{P+1} \subseteq \mathcal{O}_{r_P}^P,
\end{equation}
it must be that there exists $(\theta,\eta) \in \mathbb{S}^1 \times \mathcal{P}(\mathbb{S}^1)$ such that $\bigcap_{P\geq 1} \mathcal{O}_{r_P}^P = (\theta,\eta)$.  For any $(\theta,\mu) \in \mathbb{S}^1 \times \mathcal{P}(\mathbb{S}^1)$, since if $(\theta,\mu) \in \mathcal{O}^{P+1}_i$ and $(\theta,\mu) \in \mathcal{O}^{P}_j$, it must be that $\mathcal{Z}^{P+1}_i \subseteq \mathcal{Z}^P_j$, and therefore
\begin{equation}
\sum_{i=1}^P \chi\lbrace (\theta,\mu) \in \mathcal{O}^P_i \rbrace \inf_{\gamma \in \mathcal{Z}^P_i}I_\theta(\gamma) \leq \sum_{i=1}^{P+1} \chi\lbrace (\theta,\mu) \in \mathcal{O}^{P+1}_i \rbrace\inf_{\gamma \in \mathcal{Z}^{P+1}_i}I_\theta(\gamma).
\end{equation}
It thus follows from the monotone convergence theorem that
\begin{align}
\lim_{P\to\infty} \breve{I}_{\mathcal{O}^P_1,\ldots,\mathcal{O}^P_P}\big( (\nu(\mathcal{O}^P_1),\ldots, \nu(\mathcal{O}_P))\big) = \mathbb{E}^{\nu}\big[\lim_{P\to\infty}\sum_{i=1}^P\big( \chi\lbrace (\theta,\mu) \in \mathcal{O}^P_i \rbrace \inf_{\gamma \in \mathcal{Z}^P_i}I_\theta(\gamma) \big) \big],
\end{align}
and the left-hand-side is non-decreasing in $P$. It is immediate from the definition that
\begin{equation}
 \mathbb{E}^{\nu}\big[\lim_{P\to\infty}\sum_{i=1}^P \chi\lbrace (\theta,\mu) \in \mathcal{O}^P_i \rbrace \inf_{\gamma \in \mathcal{Z}^P_i}I_\theta(\gamma)  \big] \leq \mathbb{E}^{\nu}[I_\theta(\mu)].
\end{equation}
Since $I_\theta(\mu)$ is continuous in $\theta$, and lower-semi-continuous in $\mu$, it must be that 
\begin{equation}
 \mathbb{E}^{\nu}\big[\lim_{P\to\infty}\sum_{i=1}^P \big(\chi\lbrace (\theta,\mu) \in \mathcal{O}^P_i \rbrace \inf_{\gamma \in \mathcal{Z}^P_i}I_\theta(\gamma) \big) \big] \geq \mathbb{E}^{\nu}[I_\theta(\mu)].
\end{equation}
We thus see that
\begin{equation}
\lim_{P\to\infty} \breve{I}_{\mathcal{O}^P_1,\ldots,\mathcal{O}^P_P}\big( (\nu(\mathcal{O}^P_1),\ldots, \nu(\mathcal{O}^P_P))\big) = \mathbb{E}^{\nu}[I_\theta(\mu)].
\end{equation}
We have thus proved that  $\mathbb{E}^{\nu}[I_\theta(\mu)] \geq \breve{I}(\nu) - \epsilon$. It is clear from the definition of $\breve{I}$ that 
\[
\mathbb{E}^{\nu}[I_\theta(\mu)] \leq \breve{I}(\nu).
\]
Taking $\epsilon \to 0$, it must be that 
\begin{equation}
 \mathbb{E}^{\nu}[I_\theta(\mu)] = \breve{I}(\nu).
\end{equation}
\end{proof}
\subsection{LDP for the conditional empirical measure}
Throughout this section, suppose that there exists a partition of $\mathbb{S}^1$ into intervals, i.e. $\bigcup_{i=1}^Q R_i$ , and such that for each $\alpha \in \mathbb{S}^1$,
\begin{equation} \label{eq: connection probabilities piecewise constant}
\mathcal{C}(\theta_1,\alpha) = \mathcal{C}(\theta_2,\alpha)
\end{equation}
if $\theta_1,\theta_2 \in R_i$ for some $R_i$. We also assume uniform lower and upper bounds, i.e.
\begin{align}
\inf_{\theta,\alpha\in\mathbb{S}^1}\mathcal{C}(\alpha,\theta) >& 0 \\
\sup_{\theta,\alpha\in\mathbb{S}^1}\mathcal{C}(\alpha,\theta) <& \infty.
\end{align}
Let $ \mathcal{M}^+( \mathbb{S}^1 \times \mathcal{E} )$ be the set of all positive Borel measures on $\mathbb{S}^1 \times \mathcal{E}$. Define
\[
\tilde{\mu}^n_{j} = \frac{1}{\rho_n(2n+1)}\sum_{k \in I_n}\delta_{(\theta^k_n , w^{jk})} \in \mathcal{M}^+( \mathbb{S}^1 \times \mathcal{E} ),
\]
and observe that 
\begin{align} \label{eq: hat mu n j change}
\breve{\mu}^n_j &= \pi \cdot \tilde{\mu}^n_j \text{ where }
\pi :  \mathcal{M}^+\big( \mathbb{S}^1 \times \mathcal{E} \big) \to \mathcal{P}(\mathbb{S}^1) \text{ is such that } \\
(\pi \cdot \nu)(B) &= \frac{\nu(\theta\in B \text{ and }w =1)}{\nu(w=1)} \text{ for }B \in \mathcal{B}(\mathbb{S}^1). \label{eq: pi map}
\end{align}
Since $w^{jj} = 1$ always, $\pi \cdot \tilde{\mu}^n_j$ is always well-defined. We equip $ \mathcal{M}^+( \mathbb{S}^1 \times \mathcal{E} )$ with the topology $\tilde{\mathcal{T}}$ generated by open sets of the form
\begin{equation}\label{eq: Topology M S1 E}
\big\lbrace \mu \in \mathcal{M}^+\big( \mathbb{S}^1 \times \mathcal{E} \big) : \big|  \mathbb{E}^{\mu}[ g(\theta) w] -x \big| < \delta \big\rbrace,
\end{equation}
for $x,\delta \in \mathbb{R}$ and continuous $g: \mathbb{S}^1 \to \mathbb{R}$. Measures with identical expectations for all such functions are identified. Observe that $\pi$ is uniformly continuous over sets of the form $\lbrace \nu : \nu(w=1) \geq \epsilon > 0 \rbrace$. Let $\tilde{\Pi}^n_j \in \mathcal{P}\big( \mathcal{M}^+(\mathbb{S}^1 \times \mathcal{E}) \big)$ be the probability law of $\tilde{\mu}^n_j$. 
\begin{lemma}
Let $\lbrace j_n \rbrace_{n\geq 1}$ be any sequence such that there exists $p \geq 1$ such that $\theta^{j_n}_n \in R_p$ for all $n\geq 1$, and $\theta^{j_n}_n \to \alpha \in R_p$. For any open $O \in \mathcal{B}(\mathcal{P}(\mathbb{S}^1))$ and closed $F \in \mathcal{B}(\mathcal{P}(\mathbb{S}^1))$ (according to the topology generated by open sets of the form \eqref{eq: Topology M S1 E}),
\begin{align}
\linf{n}(\rho_n (2n+1))^{-1} \log \mathbb{P}\big( \breve{\mu}^n_{j_n} \in O \big) &\geq - \inf_{\mu \in O}I_{\alpha}(\mu) \label{eq: LDP lower bound intermediate} \\
\lsup{n}(\rho_n (2n+1))^{-1} \log \mathbb{P}\big( \breve{\mu}^n_{j_n} \in F \big) &\leq - \inf_{\mu \in F}I_{\alpha}(\mu).
\end{align}
\end{lemma}
\begin{proof}
Fix $\epsilon > 0$ and define $\mathcal{U}_\epsilon = \big\lbrace \mu \in \mathcal{M}^+(\mathbb{S}^1\times\mathcal{E}) \; : \; \mu(w=1) \leq \epsilon \big\rbrace $. By a union-of-events bound
\begin{multline*}
\linf{n}(\rho_n (2n+1))^{-1} \log \mathbb{P}\big( \breve{\mu}^n_{j_n} \in O \big) =\max\big\lbrace  \linf{n}(\rho_n (2n+1))^{-1} \log \mathbb{P}\big( \breve{\mu}^n_{j_n} \in O \text{ and }\tilde{\mu}^n_{j_n} \notin \mathcal{U}_\epsilon \big) , \\
 \linf{n}(\rho_n (2n+1))^{-1} \log \mathbb{P}\big( \tilde{\mu}^n_{j_n} \in \mathcal{U}_\epsilon \big) \big\rbrace .
\end{multline*}
Now
\begin{align*}
 \linf{n}(\rho_n (2n+1))^{-1} \log \mathbb{P}\big( \breve{\mu}^n_{j_n} \in O \text{ and }\tilde{\mu}^n_{j_n} \notin \mathcal{U}_\epsilon \big)
 &= \linf{n}(\rho_n (2n+1))^{-1} \log \mathbb{P}\big( \tilde{\mu}^n_{j_n} \in \pi^{-1}(O) \text{ and }\tilde{\mu}^n_{j_n} \notin \mathcal{U}_\epsilon \big) \\
 &\geq - \inf_{\mu \in (\pi^{-1}O) \cap \mathcal{U}^c_\epsilon}\tilde{I}_\alpha(\mu),
\end{align*}
using the Large Deviation Principle of Lemma \ref{Lemma LDP big space}, since $\pi^{-1}(O) \cap \mathcal{U}^c_\epsilon$ is open. 

Suppose first that (noting that $\pi^{-1}(O)$ is always non-empty)
\begin{equation}\label{eq: infimum infinity assmpt 1}
\inf_{\nu \in \pi^{-1}(O)} \tilde{I}_\alpha(\nu) = \infty.
\end{equation}
In this case, if $\nu \in \mathcal{M}^+(\mathbb{S}^1 \times \mathcal{E})$ is such that $\pi \cdot \nu =\mu \in O$, then it must be that either $\nu$ does not have a density, or it has a density $\zeta(\theta)$ and
\[
\int_{\mathbb{S}^1}\zeta(\theta)\log(\zeta(\theta))d\theta = \infty.
\] 
Whichever of these cases holds, it must be that 
\begin{equation}
I_{\alpha}(\mu) = \frac{1}{2\pi}\int_{\mathbb{S}^1}\mathcal{C}(\alpha,\theta)d\theta.\label{eq: I alpha mu temproary}
\end{equation}
Furthermore it follows from Lemma \ref{Lemma U m limit} that
\begin{align*}
\lim_{\epsilon \to 0} \linf{n}(\rho_n (2n+1))^{-1} \log \mathbb{P}\big( \tilde{\mu}^n_{j_n} \in \mathcal{U}_\epsilon \big) &= -\frac{1}{2\pi}\int_{\mathbb{S}^1}\mathcal{C}(\alpha,\theta)d\theta \\
 &=-\inf_{\mu \in O}I_{\alpha}(\mu) ,
\end{align*}
using \eqref{eq: I alpha mu temproary}. We thus obtain \eqref{eq: LDP lower bound intermediate} in the case that \eqref{eq: infimum infinity assmpt 1} holds, as required.

We now suppose that $\inf_{\nu \in \pi^{-1}(O)} \tilde{I}_\alpha(\nu) < \infty$. Then we first claim that
\begin{equation}\label{eq: tilde I nu minimum}
\inf_{\nu \in \pi^{-1}(O) \cap \mathcal{U}^c_\epsilon}\tilde{I}_\alpha(\nu) \to \inf_{\nu \in \pi^{-1}(O)} \tilde{I}_\alpha(\nu).
\end{equation}
To see this, let $\tilde{\nu}\in \pi^{-1}(O) $ be such that, for some $\delta > 0$,
\[
\tilde{I}_{\alpha}(\tilde{\nu}) \leq  \inf_{\nu \in \pi^{-1}(O)} \tilde{I}_\alpha(\nu) + \delta.
\]
The openness of $\pi^{-1}(O)$ implies that there must exist $\epsilon > 0$ such that $\nu \in  \pi^{-1}(O) \cap \mathcal{U}^c_\epsilon$. Since $\delta$ is arbitrary, we have established \eqref{eq: tilde I nu minimum}.

By Lemma \ref{Lemma Infimum Rate Function}, 
\[
 \inf_{\nu \in \pi^{-1}(O)} \tilde{I}_\alpha(\nu) = \inf_{\mu \in O}I_{\alpha}(\mu).
\]
Now $I_\alpha(\mu) \leq \frac{1}{2\pi}\int_{\mathbb{S}^1} \mathcal{C}(\alpha,\theta)d\theta$, and we thus find from \eqref{eq: epsilon to zero U} that
\begin{multline*}
\lim_{\epsilon \to 0}\max\big\lbrace  \linf{n}(\rho_n (2n+1))^{-1} \log \mathbb{P}\big( \breve{\mu}^n_{j_n} \in O \text{ and }\tilde{\mu}^n_{j_n} \notin \mathcal{U}_\epsilon \big) , \\
 \linf{n}(\rho_n (2n+1))^{-1} \log \mathbb{P}\big( \tilde{\mu}^n_{j_n} \in \mathcal{U}_\epsilon \big) \big\rbrace \\  =  \lim_{\epsilon \to 0}\linf{n}(\rho_n (2n+1))^{-1} \log \mathbb{P}\big( \breve{\mu}^n_{j_n} \in O \text{ and }\tilde{\mu}^n_{j_n} \notin \mathcal{U}_\epsilon \big) 
\geq -  \inf_{\mu \in O}I_{\alpha}(\mu).
\end{multline*}
We now turn to the upper bound. Fix $\epsilon > 0$ and define $\mathcal{V}_\epsilon = \big\lbrace \mu \in \mathcal{M}^+(\mathbb{S}^1\times\mathcal{E}) \; : \; \mu(w=1)\geq \epsilon \big\rbrace $. By a union-of-events bound
\begin{multline*}
\lsup{n}(\rho_n (2n+1))^{-1} \log \mathbb{P}\big( \breve{\mu}^n_{j_n} \in F \big) \leq   \lsup{n}(\rho_n (2n+1))^{-1} \log \mathbb{P}\big( \breve{\mu}^n_{j_n} \in F \text{ and }\tilde{\mu}^n_{j_n} \in \mathcal{V}_\epsilon \big) \\
=   \lsup{n}(\rho_n (2n+1))^{-1} \log \mathbb{P}\big( \tilde{\mu}^n_{j_n} \in \pi^{-1}(F) \text{ and }\tilde{\mu}^n_{j_n} \in \mathcal{V}_\epsilon \big) 
\leq -\inf_{\nu \in \pi^{-1}(F) \cap \mathcal{V}_{\epsilon}}I_\alpha(\nu)
\end{multline*}
If $\inf_{\nu \in \pi^{-1}(F) }I_\alpha(\nu) = \infty$, then it immediately follows that
\begin{equation*}
\lsup{n}(\rho_n (2n+1))^{-1} \log \mathbb{P}\big( \breve{\mu}^n_{j_n} \in F \big) \leq  - \inf_{\mu \in F}I_\alpha(\nu).
\end{equation*}
Otherwise, suppose that $\inf_{\nu \in \pi^{-1}(F) }I_\alpha(\nu) < \infty$. For any $\mu \in F$ and any $a> 0$, $a\mu \in \pi^{-1}F$. We thus find from Lemma \ref{Lemma Infimum Rate Function} that 
\[
\inf_{\nu \in \pi^{-1}(F)}\tilde{I}_\alpha(\nu) = \inf_{\mu \in F}I_\alpha(\mu).
\]
\end{proof}
The following Lemma establishes an `exponential tightness' property.
\begin{lemma}\label{Lemma Exponential Tightness}
Let $\lbrace j_n \rbrace_{n\geq 1}$ be any sequence such that there exists $p$ such that $\theta^{j_n}_n \in R_p$ for all $n\geq 1$. For any $M > 0$, there exists a subset $\mathcal{U}_m \subset  \mathcal{M}^+\big( \mathbb{S}^1 \times \mathcal{E} \big)$, where for $m \geq 0$,
\begin{equation}\label{eq: U m definition}
\mathcal{U}_m = \big\lbrace \mu \in \mathcal{M}^+(\mathbb{S}^1\times\mathcal{E}) \; : \; \mu(w=1) \leq m \big\rbrace
\end{equation}
 such that 
\begin{equation}
\lsup{n}(\rho_n (2n+1))^{-1} \log \mathbb{P}\big( \tilde{\mu}^n_{j_n} \notin \mathcal{U}_m \big) \leq -M
\end{equation}
\end{lemma}
\begin{proof}
By Chernoff's Inequality, for any $a > 0$, 
\begin{align}
\mathbb{P}\big( \tilde{\mu}^n_{j_n} \notin \mathcal{U}_m \big) \leq & \exp\big( -\rho_n (2n+1)am \big) \mathbb{E}\big[\exp\big(a  \sum_{k\in I_n}w^{j_n,k} \big) \big] \\
=& \exp\big( -\rho_n (2n+1)am \big)\prod_{k\in I_n}\big\lbrace 1+ \rho_n \mathcal{C}(\theta^{j_n}_n , \theta^{k}_n)\big( \exp(a) - 1 \big) \big\rbrace.
\end{align}
Similarly to the proof of Lemma \ref{Lemma LDP big space}, we find that 
\begin{equation}
\lsup{n}\big(\rho_n (2n+1)\big)^{-1} \log \prod_{k\in I_n}\big\lbrace 1+ \rho_n \mathcal{C}(\theta^{j_n}_n , \theta^{k_n}_n)\big( \exp(a) - 1 \big) \big\rbrace \to \frac{1}{2\pi} \big( \exp(a) - 1 \big)\int_{\mathbb{S}^1} \mathcal{C}(\alpha,\beta)d\beta .
\end{equation}
Taking $a=1$ and $m$ to be sufficiently large, the previous two equations imply that 
\begin{equation}
\lsup{n}\big\lbrace \rho_n(2n+1) \big\rbrace \log \mathbb{P}\big( \tilde{\mu}^n_{j_n} \notin \mathcal{U}_m \big) \leq - M.
\end{equation}
\end{proof}
\begin{lemma}\label{Lemma Exact Asymtotic LDP}
Let $\lbrace j_n \rbrace_{n\geq 1}$ be any sequence such that there exists $p$ such that $\theta^{j_n}_n \in R_p$ for all $n\geq 1$, and also $\theta^{j_n}_n \to \alpha \in R_p$. Let $\mathcal{O} \subset\mathcal{P}(\mathbb{S}^1)$ be open and let $\bar{\mathcal{O}}$ denote its closure. Then
\begin{align}
\lim_{n\to\infty}  \lbrace (2n+1)\rho_n \rbrace^{-1}\log \mathbb{P}\big(  \hat{\mu}^n_{j_n} \in \mathcal{O} \big) 
&=\lim_{n\to\infty}  \lbrace (2n+1)\rho_n \rbrace^{-1}\log \mathbb{P}\big(  \hat{\mu}^n_{j_n} \in \bar{\mathcal{O}} \big)\nonumber \\
 &= - \inf_{\mu \in \mathcal{O}} I_\alpha(\mu) = - \inf_{\mu \in \bar{\mathcal{O}}} I_\alpha(\mu)   . \label{eq: Log Moment Generating Function H n j}
\end{align}
\end{lemma}
\begin{proof}
For $1\leq p \leq M$, let $\mathcal{O}_p \subseteq \mathcal{P}( \mathbb{S}^1 )$ be any set of the form
\begin{equation}
\mathcal{O}_p = \big\lbrace\mu \in \mathcal{P}( \mathbb{S}^1 ) \; : \; \mu( F^p_i ) > \lambda^p_i \text{ for all }1\leq i \leq m_p \big\rbrace, 
\end{equation}
for some $m_p\in \mathbb{Z}^+$, constants $\lambda^p_i \in [0,1)$ such that
\begin{equation}\label{eq: sum lambda}
\sum_{i=1}^{m_p} \lambda^p_i < 1,
\end{equation}
and closed mutually disjoint sets $F^p_i \subset \mathbb{S}^1$, i.e. $F^p_i \cap F^p_j = \emptyset \text{ if }i\neq j$. It follows from the Portmanteau Theorem that $\mathcal{O}_p$ is open \cite{Billingsley1999}.

We now claim that one must be able to find a countably infinite set of such open sets such that
\begin{equation}\label{eq: countably infinite covering}
\mathcal{O} = \cup_{p\geq 1} \mathcal{O}_p.
\end{equation}
This is because $\mathcal{P}(\mathbb{S}^1)$ is separable with respect to the weak topology, and open sets of the type $\mathcal{O}_p$ generate the weak topology \cite{Billingsley1999}. It follows from \eqref{eq: countably infinite covering} that
\begin{equation}\label{eq: countably infinite covering 2}
\bar{\mathcal{O}} = \cup_{p\geq 1} \bar{\mathcal{O}}_p.
\end{equation}
Using the Large Deviations estimate (in the second line) and Lemma \ref{Lemma Log Moment } in the penultimate line,
\begin{align}
\lsup{n} \big\lbrace (2n+1)\rho_n \big\rbrace^{-1}\log \mathbb{P}\big( \hat{\mu}^n_{j_n} \in \mathcal{O}  \big)  
&\leq \lim_{n\to\infty}  \big\lbrace (2n+1)\rho_n \big\rbrace^{-1}\log \mathbb{P}\big(  \hat{\mu}^n_{j_n} \in \bar{\mathcal{O}}  \big)\nonumber\\  &= - \inf_{\mu \in \bar{\mathcal{O}}} I_\alpha(\mu) 
= - \inf_{p\geq 1}\inf_{\mu \in \bar{\mathcal{O}}_p} I_\alpha(\mu)\nonumber \\ &= - \inf_{p\geq 1}\inf_{\mu \in \mathcal{O}_p} I_\alpha(\mu) = - \inf_{\mu \in \mathcal{O}}I_\alpha(\mu).
\end{align}
The Large Deviations estimate also implies that
\begin{equation}
\linf{n} \big\lbrace (2n+1)\rho_n \big\rbrace^{-1}\log \mathbb{P}\big(  \hat{\mu}^n_{j_n} \in \mathcal{O}  \big)  \geq - \inf_{\mu \in \mathcal{O}}I_\alpha(\mu).
\end{equation}
We have thus established the lemma.
\end{proof}

\begin{lemma}\label{Lemma Log Moment }
Let $\mathcal{O} \subseteq \mathcal{P}( \mathbb{S}^1 )$ be any set of the form
\begin{equation}
\mathcal{O} = \big\lbrace\mu \in \mathcal{P}( \mathbb{S}^1 ) \; : \; \mu( F_i ) > \lambda_i \text{ for all }1\leq i \leq m \big\rbrace, 
\end{equation}
for some $m\in \mathbb{Z}^+$, constants $\lambda_i \in [0,1)$ such that
\begin{equation}\label{eq: sum lambda}
\sum_{i=1}^m \lambda_i < 1,
\end{equation}
and closed mutually disjoint sets $F_i \subset \mathbb{S}^1$, i.e. $F_i \cap F_j = \emptyset \text{ if }i\neq j$. Then
\begin{equation}\label{eq: rate function continuity}
\inf_{\mu \in \mathcal{O}}I_{\alpha}(\mu) = \inf_{\mu \in \bar{\mathcal{O}}}I_{\alpha}(\mu) .
\end{equation}

\end{lemma}
\begin{proof}
Since $I_{\alpha}$ is lower semi-continuous, there exists $\nu \in \bar{\mathcal{O}}$ such that 
\[
 \inf_{\mu \in \bar{\mathcal{O}}}I_{\alpha}(\mu) = I_\alpha(\nu).
\]
Since $I_\alpha(\gamma) \leq \frac{1}{2\pi}\int_{\mathbb{S}^1} \mathcal{C}(\alpha,\theta)d\theta $ for all $\gamma \in \mathcal{P}(\mathbb{S}^1)$, we may assume that
\[
 I_\alpha(\nu) < \frac{1}{2\pi}\int_{\mathbb{S}^1} \mathcal{C}(\alpha,\theta)d\theta ,
\]
since otherwise the lemma is obviously true. The definition of $I_\alpha$ implies that $\nu$ has a density $\zeta(\theta) / (2\pi)$. It follows from the Portmanteau Theorem that
\begin{equation}
\nu(F_i) \geq \lambda_i \text{ for all }1\leq i \leq m
\end{equation}
and therefore
\begin{equation}\label{eq: zeta lambda upper bound}
\frac{1}{2\pi}\int_{F_i}\zeta(\theta)d\theta \geq \lambda_i.
\end{equation}
Notice that 
\begin{align}
I_\alpha(\nu) =&  \frac{1}{2\pi}\int_{\mathbb{S}^1} \mathcal{C}(\alpha,\theta)d\theta - \exp\big(-R(\nu) \big) \text{ where }\\
R(\nu) =& \frac{1}{2\pi} \int_{\mathbb{S}^1} \zeta(\theta)\log\big( \zeta(\theta) / \mathcal{C}(\alpha,\theta) \big)d\theta .
\end{align}
Define $\upsilon \in \mathcal{O}$ to be such that $\upsilon$ has a density $\phi(\theta) / (2\pi)$ and also satisfies the following properties:
\begin{align}
\upsilon\big( \cup_{i} F_i \big) =& 1 \\
\phi &\text{ is constant over each }F_i \\
\upsilon( F_i) =& \lambda_i / (\sum_{q=1}^m \lambda_q).\label{eq: upsilon F i}
\end{align}
Next, for $q\in \mathbb{Z}^+$, define the measure $\upsilon^q = (1-q^{-1})\nu +q^{-1}\upsilon$, and observe that $\upsilon^q$ has a density $\zeta^q$ given by $\zeta^q(\theta) = (1-q^{-1})\zeta(\theta) + q^{-1}\phi(\theta)$. It follows from \eqref{eq: zeta lambda upper bound} and \eqref{eq: upsilon F i} that $\upsilon^q \in \mathcal{O}$. It thus remains for us to demonstrate that $\lim_{q\to\infty}I_{\alpha}(\upsilon^q) = I_\alpha(\nu)$.

Let $A \subset \mathbb{S}^1$ be the set
\begin{equation}
A = \big\lbrace  \theta : \zeta(\theta) \geq \phi(\theta) \text{ and }\zeta(\theta) \geq \mathcal{C}(\alpha,\theta) \big\rbrace.
\end{equation}
We have that 
\begin{align*}
R(\upsilon^q) =  \frac{1}{2\pi} \int_{A} \zeta^q(\theta)\log\big( \zeta^q(\theta) / \mathcal{C}(\alpha,\theta) \big)d\theta +   \frac{1}{2\pi} \int_{\mathbb{S}^1 / A} \zeta^q(\theta)\log\big( \zeta^q(\theta) / \mathcal{C}(\alpha,\theta) \big)d\theta.
\end{align*}
Thanks to the dominated convergence theorem, since $\zeta^q(\theta) \to \zeta(\theta)$ as $q\to\infty$,
\[
\lim_{q\to\infty} \int_{\mathbb{S}^1 / A} \zeta^q(\theta)\log\big( \zeta^q(\theta) / \mathcal{C}(\alpha,\theta) \big)d\theta = \int_{\mathbb{S}^1 / A}  \zeta(\theta)\log\big( \zeta(\theta) / \mathcal{C}(\alpha,\theta) \big)d\theta .
\]
 Now one easily checks that for $\theta \in A$,
\[
\zeta^{q+1}(\theta)\log\big( \zeta^{q+1}(\theta) / \mathcal{C}(\alpha,\theta) \big) \geq \zeta^q(\theta)\log\big( \zeta^q(\theta) / \mathcal{C}(\alpha,\theta) \big),
\]
since the function $x \mapsto x\log(x)$ is increasing for $x \geq 1$. We thus find from the monotone convergence theorem that
\[
\lim_{q\to\infty} \int_{A} \zeta^q(\theta)\log\big( \zeta^q(\theta) / \mathcal{C}(\alpha,\theta) \big)d\theta =    \frac{1}{2\pi} \int_{ A} \zeta(\theta)\log\big( \zeta(\theta) / \mathcal{C}(\alpha,\theta) \big)d\theta .
\]
We can thus conclude that $\lim_{q\to\infty}R(\upsilon^q) = R(\nu)$, and therefore $\lim_{q\to\infty}I_{\alpha}(\upsilon^q) = I_\alpha(\nu)$. We have established \eqref{eq: rate function continuity}.

\end{proof}

\begin{lemma}\label{Lemma LDP big space}
Let $\lbrace j_n \rbrace_{n\geq 1}$ be any sequence such that there exists $p$ such that $\theta^{j_n}_n \in R_p$ for all $n\geq 1$, and also $\theta^{j_n}_n \to \alpha \in R_p$. Then $\lbrace \tilde{\Pi}^n_{j_n} \rbrace_{n\geq 1}$  satisfy a Large Deviation Principle with rate function
\begin{align}
\tilde{I}_{\alpha}(\nu) &= \left\lbrace \begin{array}{c}
\infty  \text{ if }\nu(\cdot ,1) \text{ does not have a density, otherwise} \\
 \frac{1}{2\pi}\int_{\mathbb{S}^1}\big\lbrace \gamma(\theta) \log\big(\frac{\gamma(\theta)}{C(\alpha,\theta)} \big) - \gamma(\theta) + \mathcal{C}(\alpha,\theta)\big\rbrace d\theta \text{ if }\nu(d\theta,1)= \frac{\gamma(\theta)}{2\pi}d\theta\\
\end{array}\right. \label{eq: I alpha nu}.
\end{align}
The LDP means that for open sets $O \subset \mathcal{M}^+(\mathbb{S}^1\times \mathcal{E} )$ and closed sets $F \subset \mathcal{M}^+(\mathbb{S}^1\times\mathcal{E} )$,
\begin{align}
\lsup{n}\frac{1}{\rho_n(2n+1)^2} \log \tilde{\Pi}^n_{j_n}(F) &\leq -\inf_{\nu \in F}\tilde{I}_\alpha(\nu) \label{eq: breve I nu 1} \\
\linf{n}\frac{1}{\rho_n(2n+1)^2} \log \tilde{\Pi}^n_{j_n}(O)&\geq -\inf_{\nu \in O}\tilde{I}_\alpha(\nu) .\label{eq: breve I nu 2}
\end{align}
$\tilde{I}$ is lower semicontinuous and has compact level sets.
\end{lemma}
\begin{proof}
We first show that the LDP (relative to the topology $\tilde{\mathcal{T}}$ defined in \eqref{eq: Topology M S1 E}) holds with the rate function
\begin{align}
\tilde{I}_{\alpha}(\nu) &= \sup_{ h \in \mathcal{C}_b(\mathbb{S}^1)} \big\lbrace \mathbb{E}^{\nu}[h(\theta)w] - \tilde{\Lambda}(\alpha,h) \big\rbrace , \label{eq: breve I rate function} \\
\tilde{\Lambda}(\alpha,h) &= \frac{1}{2\pi} \int_{-\pi}^{\pi} \mathcal{C}(\alpha,\beta)\big(\exp\lbrace h(\beta)\rbrace-1\big) d\beta
\end{align}
Define the logarithmic moment generating function, for any bounded and continuous $h : \mathbb{S}^1 \to \mathbb{R}$,
\begin{equation}
\breve{\Lambda}(\alpha,h) = \lsup{n}\big\lbrace \rho_n (2n+1) \big\rbrace^{-1}\log \mathbb{E}\big[\exp\big( \sum_{k\in I_n}h(\theta^k_n)w^{j_n k} \big) \big],
\end{equation}
recalling that $ \lbrace j_n \rbrace_{n=1}^\infty $ is any sequence such that $\theta^j_n \to \alpha \in \mathbb{S}^1$. We first establish that for all such sequences,
\begin{equation} \label{eq: to prove Lambda H a b}
\breve{\Lambda}(\alpha,h) = \tilde{\Lambda}(\alpha,h).
\end{equation}
Observe that
\begin{multline*}
 \mathbb{E}\big[\exp\big( \sum_{k=-n}^n h(\theta^k_n)w^{j_n k}  \big)\big] = \mathbb{E}\big[\exp\big(h(\theta^{j_n}_n)+ \sum_{k \neq j_n}w^{j_nk}h(\theta^k_n)w^{j_n k} \big\rbrace \big)\big]\\   
 =\exp\big(h(\theta^{j_n}_n)\big)\prod_{k\neq j_n} \big\lbrace 1 +\rho_n\mathcal{C}\big(\theta^{j_n}_n,\theta^k_n\big)\big(\exp\big(h(\theta^k_n) \big)-1\big) \big\rbrace 
\end{multline*}
since $w^{jj} = 1$ identically. We thus find that
\begin{multline}
\lim_{n\to\infty}\big(\rho_n (2n+1)\big)^{-1} \log \mathbb{E}\big[\exp\big( \sum_{k \neq j_n} h(\theta^k_n)w^{j_n k}   \big)\big]=\\
\lim_{n\to\infty}\big(\rho_n (2n+1)\big)^{-1} \sum_{k\in I_n}\log \big\lbrace 1 +\rho_n\mathcal{C}(\theta^j_n,\theta^k_n)\big(\exp\big(h(\theta^k_n) \big)-1\big) \big\rbrace \\
=\lim_{n\to\infty}\big(\rho_n (2n+1)\big)^{-1} \sum_{k \in I_n}\big\lbrace \rho_n\mathcal{C}(\theta^j_n,\theta^k_n)\big(\exp\big(h(\theta^k_n) \big)-1\big) 
+ O(\rho_n^2) \big\rbrace ,
\end{multline}
through a second degree Taylor Expansion of the logarithmic function. Since $\rho_n \to 0$, the $O(\rho_n^2)$ term is asymptotically negligible, and we therefore obtain that
\begin{equation}
\lim_{n\to\infty}\big(\rho_n (2n+1)\big)^{-1} \log  \mathbb{E}\big[\exp\big( \sum_{k=-n}^n h(\theta^k_n)w^{j_n k}  \big)\big]
= \frac{1}{2\pi} \int_{-\pi}^{\pi} \mathcal{C}(\alpha,\beta)\big(\exp\big(h(\beta)\big)-1\big) d\beta . 
\end{equation}
We have thus proved \eqref{eq: to prove Lambda H a b}. It follows from \cite[Corollary 4.6.14]{Dembo1998} that the sequence of probability laws $\lbrace \tilde{\Pi}^n_{j_n} \rbrace_{n\geq 1}$ satisfy a Large Deviation Principle. This is because the function $h \to \tilde{\Lambda}(\alpha,h)$ is Gateaux Differentiable. The exponential tightness property follows from Lemma \ref{Lemma Exponential Tightness}. One can adapt the proof of Prokhorov's Theorem to prove that $\mathcal{U}_m$ is compact for any $m > 0$. We have thus established that the sequence of probability laws $\lbrace \tilde{\Pi}^n_{j_n} \rbrace_{n\geq 1}$ satisfy a Large Deviation Principle with good rate function \eqref{eq: breve I rate function}.

Next we establish that if $\nu(w=1) > 0$ and $\nu(\cdot,1)$ does not have a density, then $\tilde{I}(\nu) = \infty$. To this end, let $B(\theta,\delta) \subset \mathbb{S}^1$ be the open ball centered at $\theta$ of radius $\delta$. Let $\upsilon \in \mathcal{P}(\mathbb{S}^1)$ be Lebesgue measure on $\mathbb{S}^1$. Suppose that there exists $\theta \in \mathbb{S}^1$ such that
\begin{equation}
\lim_{\delta \to 0}\frac{\nu\big( B(\theta,\delta)\big)}{ 2\delta} = \infty.
\end{equation}
For $\delta > 0$, define the function
\[
h_{\delta}(\alpha) = \left\lbrace \begin{array}{c c}
\delta^{-1} & \text{ if }x\in B(\theta,\delta) \\
0 & \text{ if }x \notin B(\theta,\delta + \delta^2) \\
\delta^{-1}\big(1- \delta^{-2}\inf_{y\in B(\theta,\delta)}\big\lbrace | \alpha-y \mod \mathbb{S}^1 | \big\rbrace \big)& \text{ otherwise.}
\end{array}\right.
\]
Notice that $h_{\delta}(\alpha)$ is continuous and bounded. We easily check that
\begin{align*}
\lim_{\delta \to 0}\mathbb{E}^{\nu}[ h_{\delta}] = \infty \; \; , \; \;
\lim_{\delta \to 0}\big|\tilde{\Lambda}(\alpha,h_{\delta}) \big| < \infty, 
\end{align*}
which means that $\tilde{I}_\alpha(\nu) = \infty$.

Finally, we suppose that $\nu(\cdot,1)$ has a density given by $\frac{1}{2\pi}\gamma(\cdot)$. We then see that
\begin{align*}
\tilde{I}_\alpha(\nu) = \sup_{h \in \mathcal{C}_b(\mathbb{S}^1)} \frac{1}{2\pi} \int_{\mathbb{S}^1} \gamma(\theta) \big\lbrace h(\theta) - \frac{\mathcal{C}(\alpha,\theta)}{\gamma(\theta)}\big(\exp(h(\theta))-1\big)  \big\rbrace d\theta .
\end{align*}
We can find the unique supremum of the integrand by differentiating. This implies that $\tilde{I}_\alpha(\nu)$ takes the form in \eqref{eq: I alpha nu}.
\end{proof}

\begin{lemma}\label{Lemma Infimum Rate Function}
Let $\mu \in \mathcal{P}(\mathbb{S}^1)$ have a density $\zeta(\theta) / (2\pi)$. Assume that
\begin{equation}
\int_{\mathbb{S}^1} \zeta(\theta)\log(\zeta(\theta))d\theta < \infty.
\end{equation}
Then
\begin{equation}
\inf_{\nu \in \mathcal{M}^+(\mathbb{S}^1\times\mathcal{E}): \pi\cdot \nu = \mu}\tilde{I}_{\alpha}(\nu) = I_{\alpha}(\mu).
\end{equation}
\end{lemma}
\begin{proof}
It must be that
\begin{equation}
\int_{\mathbb{S}^1}\zeta(\theta)d\theta = 2\pi.
\end{equation} 
Any $\nu \in \mathcal{M}^+(\mathbb{S}^1\times\mathcal{E})$ such that $\pi\cdot \nu = \mu$ must have density $a\zeta(\theta)$, for some $a > 0$. We then find that
\begin{equation}\label{eq: breve I alpha}
\tilde{I}_{\alpha}(\nu) =  \frac{1}{2\pi}\int_{\mathbb{S}^1}\big\lbrace a\zeta(\theta) \log\big(\frac{a\zeta(\theta)}{\mathcal{C}(\alpha,\theta)} \big) - a\zeta(\theta) + \mathcal{C}(\alpha,\theta)\big\rbrace d\theta  := \Gamma(a).
\end{equation}
Differentiating with respect to $a$ to obtain the minimum, we find that
\begin{align*}
\frac{d\Gamma}{da} = \frac{1}{2\pi}\int_{\mathbb{S}^1}\big\lbrace \zeta(\theta) \log\bigg(\frac{\zeta(\theta)}{\mathcal{C}(\alpha,\theta)} \bigg) + \zeta\log(a)  \big\rbrace d\theta .
\end{align*}
We thus find that the value of $a$ at the unique critical point is
\begin{equation}
a_* = \exp\bigg(\frac{1}{2\pi}\int_{\mathbb{S}^1} \zeta(\theta) \log\big(\frac{\mathcal{C}(\alpha,\theta)} {\zeta(\theta)}\big)d\theta  \bigg).
\end{equation}

$\Gamma$ achieves a local minimum at $a_*$, because $\Gamma \to \infty$ as $a\to\infty$. We thus find that, after substituting $a = a_*$ into \eqref{eq: breve I alpha},
\begin{align*}
I_\alpha(\mu) = \frac{1}{2\pi}\int_{\mathbb{S}^1}\mathcal{C}(\alpha,\theta) d\theta - \exp\bigg(\frac{1}{2\pi}\int_{\mathbb{S}^1} \zeta(\theta) \log\bigg(\frac{C(\alpha,\theta)} {\zeta(\theta)}\bigg)d\theta  \bigg).
\end{align*}
One can double check using Jensen's Inequality that $I_\alpha(\mu) \geq 0$.

\end{proof}
\begin{lemma} \label{Lemma U m limit}
Let $\lbrace j_n \rbrace_{n\geq 1}$ be any sequence such that there exists $p$ such that $\theta^{j_n}_n \in R_p$ for all $n\geq 1$, and also $\theta^{j_n}_n \to \alpha \in R_p$. Then, recalling the definition of $\mathcal{U}_m$ in \eqref{eq: U m definition},
\begin{equation}\label{eq: epsilon to zero U}
\lim_{m\to 0^+}\lim_{n\to\infty}\lbrace \rho_n (2n+1)\rbrace^{-1} \log \mathbb{P}\big( \tilde{\mu}^n_{j_n} \in \mathcal{U}_m \big) = \frac{1}{2\pi} \int_{\mathbb{S}^1}\mathcal{C}(\alpha,\theta)d\theta.
\end{equation}
\end{lemma}
\begin{proof}
Since $\mathcal{U}_m$ is closed, by Lemma \ref{Lemma LDP big space},
\begin{equation}
\lsup{n}\lbrace \rho_n (2n+1)\rbrace^{-1} \log \mathbb{P}\big( \tilde{\mu}^n_{j_n} \in \mathcal{U}_m \big)\leq  -\inf_{\nu \in \mathcal{U}_m}\tilde{I}_\alpha(\nu)  
\end{equation}
Define $\mathcal{V}_m = \big\lbrace \nu \in \mathcal{M}^+(\mathbb{S}^1\times\mathcal{E}) : \nu(w=1) < m \big\rbrace$. Since $\mathcal{V}_m$ is open, by Lemma \ref{Lemma LDP big space},
\begin{equation}
\linf{n}\lbrace \rho_n (2n+1)\rbrace^{-1} \log \mathbb{P}\big( \tilde{\mu}^n_{j_n} \in \mathcal{V}_m \big)\geq  -\inf_{\nu \in \mathcal{V}_m}\tilde{I}_\alpha(\nu)  
\end{equation}
It follows from the previous two equations that it suffices for us to prove that
\begin{equation}
\lim_{m\to 0}\inf_{\nu \in \mathcal{U}_m}\tilde{I}_\alpha(\nu) = \frac{1}{2\pi}\int_{\mathbb{S}^1}\mathcal{C}(\alpha,\theta)d\theta.
\end{equation}
First, define $\nu_m \in \mathcal{U}_m$ to be such that $\nu(\cdot,1)$ has the uniform density $\frac{m}{2\pi}$. One easily checks that $\tilde{I}_\alpha(\nu_m) \to  \frac{1}{2\pi}\int_{\mathbb{S}^1}\mathcal{C}(\alpha,\theta)d\theta$ as $m\to 0$. We thus see that 
\begin{equation}\label{eq: upper bound tmp middle}
\lim_{m\to 0}\inf_{\nu \in \mathcal{U}_m}\tilde{I}_\alpha(\nu) \leq  \frac{1}{2\pi}\int_{\mathbb{S}^1}\mathcal{C}(\alpha,\theta)d\theta.
\end{equation}
Now suppose for a contradiction that $\nu^m \in \mathcal{U}_m$, and $\nu^m(\cdot,1)$ has a density $\zeta^m(\theta) / (2\pi)$, and that there exists $\epsilon , m_0 > 0$ such that for all $m < m_0$,
\begin{equation}
\tilde{I}_\alpha(\nu^m) \leq  \frac{1}{2\pi}\int_{\mathbb{S}^1}\mathcal{C}(\alpha,\theta)d\theta- \epsilon
\end{equation}
Thanks to \eqref{eq: upper bound tmp middle} and the definition of $\tilde{I}_\alpha$ in \eqref{eq: I alpha nu}, it must be that for all sufficiently large $m$,
\begin{equation}
 \frac{1}{2\pi}\int_{\mathbb{S}^1}\big\lbrace \zeta^m(\theta) \log\big(\frac{\zeta^m(\theta)}{C(\alpha,\theta)} \big) - \zeta^m(\theta) \big\rbrace d\theta  \leq -\epsilon.
\end{equation}
Our assumption that $\mathcal{C}$ is lower bounded implies that there exists $C_{lb} > 0$ such that
\[
\inf_{\theta,\alpha}\mathcal{C}(\alpha,\theta) \geq C_{lb} > 0.
\]
 Thus
\begin{equation}
 \frac{1}{2\pi}\int_{\mathbb{S}^1}\big\lbrace \zeta^m(\theta) \log\big(\zeta^m(\theta) \big) +\zeta^m(\theta)\big( -\log(C_{lb}) -1\big) \big\rbrace d\theta \leq -\epsilon,
\end{equation}
We thus find that
\begin{equation}
 \frac{1}{2\pi}\int_{\mathbb{S}^1}  \zeta^m(\theta) \log\big(\zeta^m(\theta) \big)   d\theta \leq m \big( \log(C_{lb})+1\big) - \epsilon .
\end{equation}
However there exists a positive constant $K$ such that
\[
x\log x \geq -Kx.
\]
Thus 
\begin{equation}
-K \frac{1}{2\pi}\int_{\mathbb{S}^1}  \zeta^m(\theta) d\theta \leq m \big( \log(C_{lb})+1\big) - \epsilon ,
\end{equation}
and therefore
\begin{equation}
-Km \leq m \big( \log(C_{lb})+1\big) - \epsilon .
\end{equation}
Taking $m\to 0$, we obtain a contradiction. We have thus proved the lemma.
\end{proof}
\section{Proof of Theorems \ref{Theorem Psi} and \ref{Theorem 1}}\label{Section Psi Definition}
To push-forward the initial conditions by the dynamics, we wish to find the coarsest topology possible that will ensure that the dynamics depends continuously on the initial condition. The weakest such topology that we could find is $\tilde{\mathcal{T}}$, which we will now specify precisely. We will first define a topology $\mathcal{T}_m$ that `knows' about the distribution of the states $\lbrace u^k_* \rbrace$ that can be reached through no more than $m$ connections. $\mathcal{T}_0$ is the standard weak topology over $\mathcal{P}(\mathcal{D})$. $\tilde{\mathcal{T}}$ will be defined to be the smallest (coarsest) topology containing $\bigcup_{m\geq 1}\mathcal{T}_m$. The basic reason why this topology will be sufficient for the map $\Psi$ to be continuous is that one knows that the push-forward of the original dynamics can be approximated arbitrarily well by employing $m$ Euler-step approximations, for large enough $m$. However $m$ Euler-steps only requires knowing the distribution of the states within $m$ edge connections of each vertex. In Section \ref{Section Topological Specification} we specify the topology $\tilde{\mathcal{T}}$, and prove the essential properties of compactness and separability. In Section \ref{Section Push Forward Map} we specify the push-forward map $\Psi$, and prove that it is continuous when $\mathcal{P}(\mathcal{D}\times\mathcal{P}(\mathcal{D}))$
is equipped with the $\tilde{\mathcal{T}}$ topology.

\subsection{Specification of the $\tilde{\mathcal{T}}$ topology on $\mathcal{P}(\mathcal{D}\times\mathcal{P}(\mathcal{D}))$} \label{Section Topological Specification}
The topology $\tilde{\mathcal{T}}$ is, by definition, the weakest topology generated by subtopologies $\lbrace \mathcal{T}_m \rbrace_{m\geq 1}$. $\mathcal{T}_m$ is generated by a map $\Phi_m$ on $\mathcal{P}(\mathcal{D}\times\mathcal{P}(\mathcal{D}))$ (to be defined below); that is, it is the weakest possible topology such that $\Phi_m$ is continuous. We then prove that the topology $\mathcal{T}_m$ is Hausdorff, separable and compact.

Recalling that $\mathcal{D}$ is a compact subset of $\mathbb{R}^d$, let $\mathcal{V}_0= \mathcal{P}(\mathcal{D})$, and for $m\geq 1$, let $\mathcal{V}_m$ be the space
\begin{equation}
\mathcal{V}_m = \mathcal{P}(\mathcal{D}\times \mathcal{V}_{m-1}).
\end{equation}
Notice that elements of $\mathcal{V}_m$ are probability measures over the space $ \mathcal{D}\times \mathcal{V}_{m-1}$. We define a Wasserstein metric on $\mathcal{V}_m$ to metrize the convergence, as follows. Define $d_0$ to be the Wasserstein metric on $\mathcal{V}_0 := \mathcal{P}(\mathcal{D})$ induced by the norm $\norm{\cdot}$ on $\mathcal{D}\subset \mathbb{R}^d$. Then define $d_1$ to be the Wasserstein metric on $\mathcal{V}_1$ induced by the metric $\norm{\cdot} + d_0(\cdot,\cdot)$ on $\mathcal{D}\times \mathcal{P}(\mathcal{D})$. Continuing in this fashion, we let $d_k$ be the Wasserstein metric on $\mathcal{V}_k$ induced by the metric $\norm{\cdot} + d_{k-1}(\cdot,\cdot)$ on $\mathcal{D}\times \mathcal{V}_{k-1}$. The weak topology $\mathcal{T}_w$ on $\mathcal{V}_m$ is that generated by the Wasserstein metric $d_m$.

It follows from Prokhorov's Theorem that $\mathcal{V}_m$ is compact as long as $\mathcal{V}_{m-1}$ is compact. Since $\mathcal{V}_0$ is compact, the following lemma must be true.
\begin{lemma}\label{Lemma V m compact}
For each $m \in \mathbb{Z}^+$, $\mathcal{V}_m$ is a compact Polish Space.
\end{lemma}

We wish to define the map $\Phi_m : \mathcal{P}\big(\mathcal{D}\times\mathcal{P}(\mathcal{D})\big) \to \mathcal{V}_{m}$, that will, in turn, specify the topology. Before we do this, we must first define $\Lambda: \mathcal{P}\big(\mathcal{D}\times\mathcal{P}(\mathcal{D})\big) \times \mathcal{D} \to \mathcal{P}\big( \mathcal{P}(\mathcal{D})\big)$ to be such that, for any $B\in \mathcal{B}\big(\mathcal{P}(\mathcal{D})\big)$,
\begin{equation}\label{eq: Lambda mu x definition}
\Lambda_{\mu,x}(B) = \left\lbrace \begin{array}{c}
 \lim_{n \to \infty}\frac{\mathbb{E}^{\mu(y,\beta)}[\chi\lbrace \norm{y-x} \leq n^{-1} \text{ and }\beta \in B \rbrace ]}{\mathbb{E}^{\mu(y,\beta)}[\chi\lbrace \norm{y-x} \leq n^{-1} \rbrace]} \text{ in the case that the limit exists, else }\\
  \chi\lbrace \kappa \in B\rbrace\text{ otherwise,}
\end{array}\right. 
\end{equation}
for some fixed $\kappa \in \mathcal{P}(\mathcal{D})$ (it does not matter how we choose $\kappa$). It follows from Levy's Downwards Theorem that for $\mu$ almost every $x$, the above limit $n \to 0$ exists, and $\Lambda_{\mu,x}$ is a regular conditional probability distribution \cite{Shiryaev2016}. However we emphasize that $\Lambda_{\mu,x}$ is a precisely defined probability measure for every $x\in \mathcal{D}$, not just for $\mu$ almost every $x$. 

Next define $\Phi_m: \mathcal{P}\big(\mathcal{D}\times\mathcal{P}(\mathcal{D})\big) \mapsto \mathcal{V}_{m}$, using the following inductive procedure. $\Phi_1$ is just the identity map. Lets assume that $\Phi_{m-1}\cdot \mu$ is well-defined: it is a probability measure over the space $\mathcal{D}\times \mathcal{V}_{m-1}$. Define $\Phi_m \cdot \mu := \nu \in \mathcal{V}_m$, where $\nu$ is such that for any continuous function $g$ in $\mathcal{C}(\mathcal{D} \times \mathcal{V}_{m-1})$, and any continuous function $h$ in $\mathcal{C}( \mathcal{P}(\mathcal{D}))$,
\begin{equation}
\mathbb{E}^{\nu}[ gh] := \mathbb{E}^{\Phi_{m-1}\cdot \mu}\big[ g(x) \mathbb{E}^{\Lambda_{\mu,x}}[h] \big].
\end{equation}
One easily checks that $\nu$ is uniquely well-defined. It thus follows from the principle of induction that we have established the following lemma.
\begin{lemma}
$\Phi_m \cdot\mu$ is uniquely well-defined for all $m\in \mathbb{Z}^+$ and all $\mu \in \mathcal{P}(\mathcal{D}\times\mathcal{P}(\mathcal{D}))$.
\end{lemma}

Define $\mathcal{T}_m$ to be the topology on $ \mathcal{P}\big(\mathcal{D}\times\mathcal{P}(\mathcal{D})\big)$ that is generated by open sets of the form
\begin{equation}
\big\lbrace \mu : \Phi_m \cdot \mu \in \mathcal{O}\big\rbrace,
\end{equation}
for some $\mathcal{O} \subset \mathcal{V}_{m}$ that is open with respect to the weak topology $\mathcal{T}_w$ on $\mathcal{V}_{m}$. Since $\Phi_m: \big(\mathcal{P}(\mathcal{D}\times\mathcal{P}(\mathcal{D})) , \mathcal{T}_w \big) \to (\mathcal{V}_m, \mathcal{T}_w)$ is not continuous, $\mathcal{T}_m$ is a strict refinement of $\mathcal{T}_w$ (notice also that $\mathcal{T}_1$ is the standard weak topology on $ \mathcal{P}\big(\mathcal{D}\times\mathcal{P}(\mathcal{D})\big)$). Let $\tilde{\mathcal{T}}$ be the coarsest topology on $ \mathcal{P}\big(\mathcal{D}\times\mathcal{P}(\mathcal{D})\big)$ containing
\[
\bigcup_{i\in \mathbb{Z}^+} \mathcal{T}_i.
\]
We next prove some essential properties of the topology $\tilde{\mathcal{T}}$.
\begin{lemma}\label{Lemma T m compact}
 $\big(  \mathcal{P}(\mathcal{D}\times\mathcal{P}(\mathcal{D})), \mathcal{T}_m\big)$ is separable and compact.
 \end{lemma}
 \begin{proof}
We prove that the following set of empirical measures is a separable base for  $\big(  \mathcal{P}(\mathcal{D}\times\mathcal{P}(\mathcal{D})), \mathcal{T}_m\big)$,
\begin{equation}\label{eq: Upsilon 1}
\Upsilon_1 = \bigg\lbrace \hat{\mu}^q = q^{-1}\sum_{j=1}^q\delta_{(y^j , \hat{\mu}^q_j)} \; : \; \hat{\mu}^q_j = \frac{1}{\Xi_j}\sum_{k\in \Xi_j}\delta_{y^k} \text{ for some }q\in\mathbb{Z}^+ \; , \lbrace y^j \rbrace_{j=1}^q \subset \mathcal{D}\cap \mathbb{Q}^d , \; \Xi_j \subset I_q\bigg\rbrace.
\end{equation}
Clearly $\Upsilon_1$ is countable. It suffices to show that
\begin{equation}\label{eq: to show Phi m}
\Phi_m \cdot \Upsilon_1 \text{ is dense in }\mathcal{V}_m, \text{ with respect to the weak topology.}
\end{equation}
Define $\lbrace \Upsilon_n \rbrace_{n=0}^m$ inductively as follows
\begin{align*}
\Upsilon_0 =& \big\lbrace \mu \in \mathcal{V}_0: \mu=  q^{-1}\sum_{1\leq j \leq q}\delta_{y^j} \in \mathcal{P}(\mathcal{D}) \text{ for some }q\in \mathbb{Z}^+ \; , \lbrace y^j \rbrace_{1\leq j \leq q}  \subset \mathcal{D}\cap \mathbb{Q}^d \big\rbrace \\
\Upsilon_n =& \big\lbrace \mu \in \mathcal{V}_n : \mu =  q^{-1}\sum_{j=1}^q\delta_{(y^j, \nu_j)} \text{ for some }q\in \mathbb{Z}^+ \; ,   \lbrace y^j \rbrace_{1\leq j \leq q}  \subset \mathcal{D}\cap \mathbb{Q}^d \; , \lbrace \nu_j \rbrace_{1\leq j \leq q} \subset \Upsilon_{n-1} \big\rbrace, \text{ for }n\geq 1.
\end{align*}
We easily check that $\Upsilon_m$ is a countable base for $\mathcal{V}_m$. Thus in order that $\big(  \mathcal{P}(\mathcal{D}\times\mathcal{P}(\mathcal{D})), \mathcal{T}_m\big)$ is separable, it suffices for us to prove that for any fixed $\nu \in \Upsilon_m$, there exists a sequence $\lbrace \mu^p \rbrace_{p\in\mathbb{Z}^+} \subset \Upsilon_1$ such that
\begin{equation}\label{eq: to establish Psi m mu p}
\Psi_m\cdot \mu^p \to \nu \text{ as }p\to\infty.
\end{equation}
By definition, we must be able to represent $\nu$ in the following way. There are measures $\lbrace \nu^n_j \rbrace_{0\leq n \leq m-1 , j\in  \hat{\mathcal{I}}^n}$, with $\nu^n_j \in \Upsilon_n$, finite index sets $\lbrace \mathcal{I}^n_k \rbrace_{0\leq n \leq m-1 , j\in  \mathbb{Z}^+} $, and $\lbrace y^n_j \rbrace \subset \mathcal{D}\cap \mathbb{Q}^d$ such that
\begin{align}
\nu &= \big| \mathcal{I}^{m}_1 \big|^{-1}\sum_{j\in \mathcal{I}^{m}_1} \delta_{(y^{m}_j , \nu^{m-1}_j)} \label{eq: nu topology T 1}\\
\nu^{n}_k &=\big| \mathcal{I}^{n}_{k} \big|^{-1}\sum_{j\in \mathcal{I}^{n}_{k}} \delta_{(y^{n}_j , \nu^{n-1}_j)} \text{ for each }1\leq n \leq m-1 \text{ and }k\in \hat{\mathcal{I}}^n \\
\nu^0_k &=\big| \mathcal{I}^{0}_{k} \big|^{-1}\sum_{j\in  \mathcal{I}^0_k}\delta_{y^0_j} \in \mathcal{P}(\mathcal{D}). \label{eq: nu topology T 3}
\end{align}
Now for each $p \in \mathbb{Z}^+$ let $\epsilon_p > 0$ be a small constant, such that $\epsilon_p \to 0$ as $p\to \infty$. Define $\lbrace \tilde{y}^n_j \rbrace \subset \mathcal{D}\cap \mathbb{Q}^d$ to be any numbers such that for all indices $0\leq n \leq m$ and $j \in \mathcal{I}^n_k$,
\begin{align}
 \tilde{y}^n_j &\neq \tilde{y}^q_k \text{ if either }n \neq q \text{ and / or }j \neq k \\
| \tilde{y}^n_j - y^n_j | &\leq \epsilon_p.
\end{align}
The fact that $\mathbb{Q}^d$ is a separable base for $\mathbb{R}^d$ implies that such a definition is always possible. Next define $\big\lbrace \kappa^p_n \big\rbrace_{n=0}^m$ to be numbers such that
\begin{align}
\kappa^p_0 &= 1 \text{ and }
\kappa^p_n = p\sum_{l=0}^{n-1}\kappa^p_l, \text{ and define }\mu^p \in \mathcal{P}(\mathcal{D}\times\mathcal{P}(\mathcal{D})) \text{ to be such that }\label{eq: kappa mass}\\
\mu^p &= \big(| \mathcal{I}^{m}_{1} |+ \sum_{n=0}^{m-1}\sum_{k\in \hat{\mathcal{I}}^n}\big|  \mathcal{I}^{n}_{k} \big| \kappa^p_n \big)^{-1}\bigg( \sum_{n=1}^{m-1} \sum_{k\in \hat{\mathcal{I}}_n} \sum_{j\in \mathcal{I}^{n}_{k}}\kappa^p_n \delta_{(\tilde{y}^{n}_j , \tilde{\nu}^{n-1}_j)} +\sum_{k\in \hat{\mathcal{I}}^0} \sum_{j\in \mathcal{I}^{0}_{k}}\kappa^p_0 \delta_{(\tilde{y}^{0}_j )}+ \sum_{j\in \mathcal{I}^{m}_1}\kappa^p_m \delta_{(\tilde{y}^{m}_j , \tilde{\nu}^{m-1}_j)} \bigg)\label{eq: kappa mass 2} \\
\tilde{\nu}^{n}_j &=\big| \mathcal{I}^{n}_{j} \big|^{-1}\sum_{k\in \mathcal{I}^{n}_{j}} \delta_{\tilde{y}^{n}_k} \in \mathcal{P}(\mathcal{D}).\label{eq: kappa mass 3}
\end{align}
Notice that $\mu^p \in \Upsilon_1$: instead of introducing the factors of $\kappa^p_n$ in front of the Dirac measures, we could have alternatively just introduced $\kappa^p_n$ copies of the Diract measure. Define $\hat{\nu}^{(p)} \in \mathcal{V}_m$ to have the same definition as $\nu$ in \eqref{eq: nu topology T 1}-\eqref{eq: nu topology T 3}, except that $y^n_j$ is replaced by $\tilde{y}^n_j$ for each set of indices $(n,j)$. Since $\epsilon_p \to 0$, it is not hard to see that
\begin{equation}
\hat{\nu}^{(p)} \to \nu \text{ as }p\to\infty.
\end{equation}
Notice that, following the definition in \eqref{eq: Lambda mu x definition},
\begin{equation}
\Lambda_{\mu^p,\tilde{y}^{n}_j} = \delta_{\tilde{\nu}^{n-1}_j}. 
\end{equation}
Lets first understand $\lim_{p\to\infty}\Phi_2 \cdot \mu^p \in \mathcal{V}_2$. It follows from the definition in \eqref{eq: kappa mass 2} that, as $p\to\infty$, since $\kappa^p_m \gg \sum_{n=0}^{m-1}\kappa^p_n$, the Dirac measures with coefficient $\kappa^p_m$ dominates those with a coefficient in $\big\lbrace \kappa^p_r \big\rbrace_{0 \leq r \leq m-1}$. This means that, since $\tilde{y}^n_j \to y^n_j$ uniformly, 
\begin{align}
\Phi_2 \cdot \mu^p &\to  \big| \mathcal{I}^m_1 \big|^{-1}\sum_{j\in \mathcal{I}^{m}_{1}} \delta_{(y^{m}_j , \breve{\nu}^{m-1}_j)} \text{ where }\\
\breve{\nu}^{m-1}_j &= \big| \mathcal{I}^{m-1}_{j} \big|^{-1}\sum_{k\in   \mathcal{I}^{m-1}_{j} }\delta_{(y^{m-1}_k , \breve{\nu}^{m-2}_k)} \text{ and }\\
  \breve{\nu}^{m-2}_k &=  \big| \mathcal{I}^{m-2}_{k} \big|^{-1}\sum_{l\in   \mathcal{I}^{m-2}_{k} }\delta_{y^{m-2}_l }.
\end{align}
Continuing this reasoning, we find that for any $m\geq 1$,
\begin{equation}
\Phi_m \cdot \mu^p = \nu,
\end{equation}
which implies \eqref{eq: to establish Psi m mu p}, as required. We have thus established  that $\big(  \mathcal{P}(\mathcal{D}\times\mathcal{P}(\mathcal{D})), \mathcal{T}_m\big)$ is separable.\\

It remains to prove that  $\big(  \mathcal{P}(\mathcal{D}\times\mathcal{P}(\mathcal{D})), \mathcal{T}_m\big)$ is compact. To demonstrate the compactness of $ \mathcal{P}(\mathcal{D}\times\mathcal{P}(\mathcal{D}))$ with respect to $\mathcal{T}_m$, suppose that for some index set $\mathcal{I}$,
\begin{equation}
\mathcal{P}(\mathcal{D}\times\mathcal{P}(\mathcal{D})) = \bigcup_{i \in \mathcal{I}} O_i,
\end{equation}
where $O_i = \Phi_m^{-1}(\mathcal{O}_i)$, and $\mathcal{O}_i$ is open in $\mathcal{V}_m$ with respect to the weak topology. It then follows from \eqref{eq: to show Phi m} that
\[
\mathcal{V}_m = \bigcup_{i\in \mathcal{I}}\mathcal{O}_i.
\]
We established in Lemma \ref{Lemma V m compact} that $\mathcal{V}_m$ is compact. Thus there exists a finite subset of $\mathcal{I}$, written $\lbrace i_p \rbrace_{p=1}^M$, such that
\[
\mathcal{V}_m = \bigcup_{p=1}^M \mathcal{O}_{i_p}.
\]
This means that
\[
\mathcal{P}\big(\mathcal{D}\times\mathcal{P}(\mathcal{D})\big) = \bigcup_{p=1}^M O_{i_p},
\]
and we have thus demonstrated the compactness of $\mathcal{P}\big(\mathcal{D}\times\mathcal{P}(\mathcal{D})\big) $ with respect to $\mathcal{T}_m$.
 \end{proof}
 Since $\tilde{\mathcal{T}}$ is generated by $\lbrace \mathcal{T}_m \rbrace_{m\geq 1}$, it does not seem surprising that it is also compact and separable.
 \begin{lemma}\label{Lemma Upsilon Dense}
 $\big( \mathcal{P}(\mathcal{D}\times\mathcal{P}(\mathcal{D}))  , \tilde{\mathcal{T}} \big)$ is compact and separable. Also $\Upsilon_1$ (as defined in \eqref{eq: Upsilon 1}) is dense in  $\big( \mathcal{P}(\mathcal{D}\times\mathcal{P}(\mathcal{D}))  , \tilde{\mathcal{T}} \big)$.
 \end{lemma}
 \begin{proof}
Define the following product topological space $\Gamma = \prod_{m=1}^{\infty} \mathcal{U}_m$, equipping $\Gamma$ with the product topology, and let each $\mathcal{U}_m$ be $\mathcal{P}(\mathcal{D}\times\mathcal{P}(\mathcal{D}))$ equipped with the topology $\mathcal{T}_m$. Since, as proved in Lemma \ref{Lemma T m compact}, $\big( \mathcal{P}(\mathcal{D}\times\mathcal{P}(\mathcal{D})), \mathcal{T}_m\big)$ is compact, it follows from Tychonoff's Theorem that $\Gamma$ is compact. Notice that $\big( \mathcal{P}(\mathcal{D}\times\mathcal{P}(\mathcal{D}))  , \tilde{\mathcal{T}} \big)$ is topologically homeomorphic to the closed subspace of $\Gamma$ consisting of elements of the form $\prod_{m=1}^{\infty} \mu$. We can thus conclude that $\big( \mathcal{P}(\mathcal{D}\times\mathcal{P}(\mathcal{D}))  , \tilde{\mathcal{T}} \big)$ is compact. The separability of  $\big( \mathcal{P}(\mathcal{D}\times\mathcal{P}(\mathcal{D}))  , \tilde{\mathcal{T}} \big)$ follows immediately from the facts that (i) $\big( \mathcal{P}(\mathcal{D}\times\mathcal{P}(\mathcal{D}))  , \mathcal{T}_m \big)$ is separable for each $m\in \mathbb{Z}^+$ and (ii) the topology $\tilde{\mathcal{T}}$ is (by definition) generated by the topologies $\lbrace \mathcal{T}_m \rbrace_{m\geq 1}$.
 \end{proof}
 
\begin{lemma}\label{Lemma tilde T topology properties}
$\big( \mathcal{P}(\mathcal{D}\times\mathcal{P}(\mathcal{D})),\tilde{\mathcal{T}}\big)$ is Hausdorff.
%
\end{lemma}
\begin{proof}
If $\mu,\nu \in \mathcal{P}(\mathcal{D}\times\mathcal{P}(\mathcal{D}))$ and $\mu\neq \nu$, there must exist disjoint neighborhoods $A,B \in \mathcal{T}_w$ (the weak topology on  $\mathcal{P}(\mathcal{D}\times\mathcal{P}(\mathcal{D}))$), such that $\mu \in A$ and $\nu \in B$. This follows from the fact that the weak topology for a space of probability measures over a Polish space is Hausforff. Since $A,B \in \tilde{\mathcal{T}}$, we have proved the lemma.
\end{proof}

\subsection{Definition of the Push-Forward Map}\label{Section Push Forward Map}

Fix $m\in \mathbb{Z}^+$. We now define a continuous map $\Psi_m :\big(\mathcal{P}(\mathcal{D}\times\mathcal{P}(\mathcal{D})),\tilde{\mathcal{T}}\big)\to \big(\mathcal{P}\big( \mathcal{C}([0,T] , \mathbb{R}^d) \big),\mathcal{T}_w\big)$ as follows. First, we specify $\Psi_m \cdot \nu$ in the case that $\nu$ is an empirical measure in $\Upsilon_1$ (as defined in \eqref{eq: Upsilon 1}), i.e. of the form, for some integer $q \geq 1$,
\begin{align}
\nu =q^{-1}\sum_{1\leq j \leq q}\delta_{(y^j , \nu_j)} \label{eq: nu empirical measure} \; , \;
\nu_j = \big|\tilde{\Xi}_j\big|^{-1}\sum_{k\in \tilde{\Xi}_j} \delta_{y^k},
\end{align}
for arbitrary $\lbrace y^j \rbrace_{j\in I_q} \subset \mathcal{D}$ and $\tilde{\Xi}_j \subset \lbrace 1,2,\ldots, q \rbrace$. Define $\lbrace y_m^j(s) \rbrace_{s\in [0,T]} \in \mathcal{C}([0,T],\mathbb{R}^d)$ as follows: for $0\leq p < m$ and $s\in [pTm^{-1} , (p+1)Tm^{-1}]$,
\begin{align}
y_m^j\big(s \big) = & y_m^j\big( pT/m \big) + (s - pTm^{-1})\big\lbrace f(y_m^j(pT/m)) + \mathbb{E}^{\nu^m_j(pT/m)}\big[G\big( y_m^j(pT/m) , \cdot \big)\big] \big\rbrace \text{ where }\nonumber\\
\nu^m_j(s) =&  \big|\tilde{\Xi}_j\big|^{-1}\sum_{k\in \tilde{\Xi}_j} \delta_{y_m^k(s)}.
\end{align}
It can be seen that the variables $y_m^j(s)$ are well-defined: one obtains the solution at the discretized time steps $\lbrace pTm^{-1} \rbrace_{0\leq p \leq m}$ by Euler-stepping, and obtains the solution at all other times through linear interpolation. We then define
\begin{align}
\Psi_m \cdot \nu =& \hat{\nu}^n_{[0,T]} \in \mathcal{P}\big(\mathcal{C}([0,T],\mathbb{R}^d) \big) \text{ where }\\
\hat{\nu}^n_{T} =& \frac{1}{2q+1}\sum_{j\in I_n}\delta_{y^j_m([0,T])}. \label{eq: Psi m definition empirical measure}
\end{align}
Notice first that $\Psi_m \cdot \nu $ is consistently defined: a permutation of the indices (remembering to also permute the connection indices) leaves both $\nu$ and $\Psi_m\cdot \nu$ unchanged.  Recall that we proved in Lemma \ref{Lemma Upsilon Dense} that empirical measures of the form in \eqref{eq: nu empirical measure} are dense in $\big(\mathcal{P}(\mathcal{D}\times\mathcal{P}(\mathcal{D})), \tilde{\mathcal{T}}\big)$: we are going to use this property to extend the above definition to arbitrary measures in $\mathcal{P}(\mathcal{D}\times\mathcal{P}(\mathcal{D}))$. To this end, we next define, for an arbitrary $\gamma \in \mathcal{V}_m$,
\begin{align}
\Psi_m \cdot \gamma &:= \lim_{p\to\infty}\Psi_m \cdot \mu_p \text{ where }
\end{align}
$\lbrace \mu_p \rbrace_{p\geq 1}$ is any sequence of empirical measures of the general form \eqref{eq: nu empirical measure} such that
\[
\Phi_m \cdot \mu_p \to \Phi_m\cdot \gamma \text{ with respect to the weak topology on }\mathcal{V}_m.
\]
\begin{lemma}\label{Lemma Psi m well defined}
 $\Psi_m : \mathcal{P}\big(\mathcal{D}\times\mathcal{P}(\mathcal{D})\big) \to \mathcal{P}\big( \mathcal{C}([0,T] , \mathbb{R}^d) \big)$ is well-defined. Also $\Psi_m$ is continuous, when $\mathcal{P}\big(\mathcal{D}\times\mathcal{P}(\mathcal{D})\big)$ is endowed with the topology $\mathcal{T}_m$, and $ \mathcal{P}\big( \mathcal{C}([0,T] , \mathbb{R}^d) \big)$ is endowed with the weak topology.
\end{lemma}
\begin{proof}
We wish to define a map $\Gamma^m : \mathcal{V}_m \to  \mathcal{P}\big( \mathcal{C}([0,T] , \mathbb{R}^d) \big)$ such that
\begin{equation}\label{eq: Psi m definition}
\Psi_m = \Gamma^m \cdot \Phi_m.
\end{equation}
 $\Gamma^m$ will be shown to be continuous with respect to the weak topology on $\mathcal{V}_m$. Define, for $0\leq k \leq m-1$, $\delta t = Tm^{-1}$ and
\begin{align}
\mathcal{W}^m_0 =& \mathcal{P} \big(\mathcal{D}\times \mathcal{P}\big(\mathcal{D}\times \mathcal{P}\big(\ldots  \times \mathcal{P}(\mathcal{D})\big)\ldots \big) \big)\label{eq: W m 0} \\
\mathcal{W}^m_k =& \mathcal{P}\big( \mathcal{Z}_k \times \mathcal{P}\big( \mathcal{Z}_k \times \mathcal{P}\big( \ldots  \times \mathcal{P}(\mathcal{Z}_k )\ldots \big)\big) \big) \text{ where }\\
\mathcal{Z}_k =& \mathcal{C}([0,k\delta t] , \mathbb{R}^d)
\end{align}
In the above definition of $\mathcal{W}^m_0$, there are $m+1$ nested $\mathcal{D}$ (i.e. there are $m$ right parentheses on the right side of \eqref{eq: W m 0}). In the definition of $\mathcal{W}^m_k$, there are $m+1-k$ nested $\mathcal{Z}_k$.
We now define a continuous map $\Gamma^m_k : \mathcal{W}^m_k \to \mathcal{W}^m_{k+1}$ as follows. Fix $\alpha \in \mathcal{W}^m_k$ and write $\beta = \Gamma^m_k \cdot \alpha$. 

 Write $\alpha$ to be the probability law of the random variables $\big(z^{(0)}_{[0,k\delta t]} , \alpha^{(0)} \big)$, where $z^{(0)}_{[0,k\delta t]} \in \mathcal{Z}_k$ and $\alpha^{(0)} \in \mathcal{P} \big( \mathcal{Z}_k \times \mathcal{P}\big( \mathcal{Z}_k \times \mathcal{P}\big( \ldots \mathcal{P}( \mathcal{Z}_k )\ldots \big)\big) \big)$ (and there are $m-k$ nested $\mathcal{Z}_k$ in this expression for $\alpha^{(0)}$). In turn, write $\alpha^{(0)}$ to be the probability law of $\big(z^{(1)}_{[0,k\delta t]} , \alpha^{(1)} \big)$, where $\alpha^{(1)} \in  \mathcal{P}\big( \mathcal{Z}_k \times \mathcal{P}\big( \mathcal{Z}_k \times \mathcal{P}\big( \ldots \mathcal{P}(\mathcal{Z}_k )\ldots \big)\big) \big)$ (there are $m-k-1$ nested $\mathcal{Z}_k$). Eventually we obtain $\alpha^{(m-k-1)} \in \mathcal{P}(\mathcal{Z}_k \times \mathcal{P}(\mathcal{Z}_k))$: write this to be the probability law of the random variables $\big(z^{(m-k)} _{[0,k\delta t]},  \alpha^{(m-k)} \big)$, with $ \alpha^{(m-k)}  \in \mathcal{P}(\mathcal{Z}_k)$. 

Now, for $0 \leq a \leq m-k$, define new random variables $y^{(a)}_{[0,(k+1)\delta t]} \in \mathcal{Z}_{k+1}$, such that for each $s\in (k\delta t, (k+1)\delta t]$,
\begin{align}
y^{(a)}_s =& z^{(a)}_{k\delta t} + (s- k\delta t)\big\lbrace f(z^{(a)}_{k\delta t}) + \int G(z^{(a)}_{k\delta t} , x_{k\delta t})d\alpha^{(a)}(x_{[0,k\delta t]})  \big\rbrace \\
y^{(a)}_s =& z^{(a)}_s \text{ for }0\leq s \leq k\delta t.
\end{align}
Define $\beta^{(m-k-1)} \in \mathcal{P}(\mathcal{Z}_{k+1})$ to be the probability law of $ (y^{(m-k-1)}_{[0,(k+1)\delta t]})$, and for $0\leq a < m-k-1$,  define $\beta^{(a)}$ to be the probability law of $(y^{(a+1)}_{[0,(k+1)\delta t]},\beta^{(a+1)})$. Finally, $\beta \in \mathcal{W}^{m}_{k+1}$ is defined to be the probability law of $\big( y^{(0)}, \beta^{(0)} \big)$.

We observe that the above definition of $\Psi_m$ is consistent with the definition in \eqref{eq: Psi m definition empirical measure} for empirical measures.  Also it can be seen that $\Gamma^m_k : \mathcal{W}^m_k \to \mathcal{W}^m_{k+1}$ is continuous when both $\mathcal{W}^m_k$ and $\mathcal{W}^m_{k+1}$ are endowed with the weak topology. Now define
\begin{equation}
\Gamma^m = \Gamma^m_m \cdot \Gamma^m_{m-1}\cdots \ldots \Gamma^m_1 \cdot \Gamma^m_0.
\end{equation}
Evidently $\Gamma^m$ is continuous as well. By definition of the topology $\mathcal{T}_m$, the mapping $\Phi_m : \big( \mathcal{P}(\mathcal{D}\times\mathcal{P}(\mathcal{D})), \mathcal{T}_m\big) \mapsto \big(\mathcal{V}_{m+1},\mathcal{T}_w \big)$ is continuous. We have thus proved the continuity of $\Psi_m$.
\end{proof}
We first study $\Psi_m$ when $\nu \in \Upsilon_1$ is an empirical measure of the form
\begin{align}
\nu = \frac{1}{2q+1}\sum_{j\in I_q}\delta_{(y^j , \nu_j)} \label{eq: nu empirical measure 22} \; \; , \; \;
\nu_j = \big|\tilde{\Xi}_j\big|^{-1}\sum_{k\in \tilde{\Xi}_j} \delta_{y^k},
\end{align}
for arbitrary $\lbrace y^j \rbrace_{j\in I_q} \subset \mathcal{D}$, and $\Xi_j \subset I_q$. Define
\begin{align}
z^j( t ) = & y^j + \int_0^t \big\lbrace f(z^j(s)) +\mathbb{E}^{\nu_j(s)}\big[G\big( z^j(s) ,\cdot \big)\big] \big\rbrace ds \text{ where }\label{eq: z dynamics definition} \\
\nu_j(s) =&  \big|\tilde{\Xi}_j\big|^{-1}\sum_{k\in \tilde{\Xi}_j} \delta_{z^k(s)}.
\end{align}
Since the functions $f$ and $G$ are Lipschitz and bounded, Picard's Theorem implies the existence of a unique solution. One easily checks using Gronwall's Inequality  that there exist constants $\lbrace C_m \rbrace_{m\geq 1}$, with $C_m > 0$, and such that $\lim_{m\to\infty}C_m = 0$ and for all such empirical measures in $\Upsilon_1$,
\begin{equation}
\sup_{j\in I_n}\sup_{s\in [0,T]}\norm{y^j(s) - z^j(s)} \leq C_m / 2.
\end{equation}
We endow $\mathcal{P}\big(\mathcal{C}([0,T],\mathbb{R}^d)\big)$ with the Wasserstein metric, given by
\begin{equation}
d_W(\mu,\nu) = \inf_{\eta} \mathbb{E}^{\eta}\big[\sup_{s\in [0,T]}\norm{y_s - z_s} \big],
\end{equation}
the infimum being taken over all measures $\eta \in \mathcal{P}\big(\mathcal{C}([0,T],\mathbb{R}^d)\times \mathcal{C}([0,T],\mathbb{R}^d)\big)$ with marginal laws $\mu$ and $\nu$. We have thus established the following lemma.
\begin{lemma}
For all $n\geq m$,
\begin{align}\label{eq: uniform convergence m}
\sup_{\nu \in \Upsilon_1}d_W( \Psi_m \cdot \nu , \Psi_n\cdot \nu ) &\leq C_m \text{ and therefore  }\\
\lim_{m\to\infty} \Psi_m \cdot \nu &= \frac{1}{2q+1}\sum_{j\in I_q}\delta_{z^j([0,T])} \text{ for all }\nu \in \Upsilon_1.
\end{align}
\end{lemma}
We are now in a position to define $\Psi \cdot \mu$ for an arbitrary $\mu \in \mathcal{P}(\mathcal{D}\times\mathcal{P}(\mathcal{D}))$. Let $\lbrace \nu^p \rbrace_{p\geq 1} \subset \Upsilon$ be a sequence of empirical measures such that $\nu^p \to \mu$ with respect to the topology $\tilde{\mathcal{T}}$ (in Lemma \ref{Lemma T m compact} we proved that $\Upsilon$ is dense in $\mathcal{P}(\mathcal{D}\times\mathcal{P}(\mathcal{D}))$ with respect to the topology $\tilde{\mathcal{T}}$). Thanks to \eqref{eq: uniform convergence m}, for any $r \in \mathbb{Z}^+$ we can find $m_r\in \mathbb{Z}^+$ such that
\begin{equation}
\sup_{\nu \in \Upsilon_1} \sup_{n\geq m_r} d_W\big( \Psi_{m_r} \cdot \nu , \Psi_n \cdot \nu \big) \leq \frac{1}{2r}.
\end{equation}
Furthermore the continuity of $\Psi_{m_r}$ implies that we can find $p_{r} \in \mathbb{Z}^+$ such that for all $s\geq p_{r}$,
\begin{equation}
d_W\big(\Psi_{m_r} \cdot \nu^{s} , \Psi_{m_r} \cdot \mu \big) \leq \frac{1}{2r}.
\end{equation}
We choose $p_r$ to be the smallest possible integer such that the above identity holds. The previous two equations imply that the sequence $\lbrace \Psi_{m_r}\cdot \nu^{p_r} \rbrace_{r\geq 1}$ is Cauchy, i.e.
\begin{align*}
\sup_{n\geq m_r}\sup_{s\geq p_{r}} d_W\big(\Psi_{m_r}\cdot \mu ,\Psi_{n} \cdot \nu^{s} \big)  \leq r^{-1}.
\end{align*}
Thus, since $\mathcal{P}\big(\mathcal{C}([0,T],\mathbb{R}^d)\big)$ is complete, the sequence must have a unique limit in $\mathcal{P}\big(\mathcal{C}([0,T],\mathbb{R}^d)\big)$. Furthermore one easily checks that the limit is independent of the choice of the approximating sequence of empirical measures $\lbrace \nu^p \rbrace_{p\in\mathbb{Z}^+}$. We define
\begin{equation}\label{eq: Psi map definition}
\Psi \cdot \mu := \lim_{r\to\infty} \Psi_{m_r}\cdot \mu 
\end{equation}

\begin{lemma}\label{Lemma Psi Limit}
The map $\Psi :\big(\mathcal{P}\big(\mathcal{D}\times\mathcal{P}(\mathcal{D})\big) , \tilde{\mathcal{T}}\big) \to \big( \mathcal{P}\big( \mathcal{C}([0,T] , \mathbb{R}^d) \big) , \mathcal{T}_w \big)$ in \eqref{eq: Psi map definition} is continuous.
\end{lemma}
\begin{proof}
This follows immediately from the definition of $\Psi$, and the continuity of $\Psi_m$.
\end{proof}
Since $\Psi\cdot \hat{\mu}^n_* = \hat{\mu}^n_T$, the above lemma implies Theorem \ref{Theorem Psi} (the other parts of this Theorem are established in Section \ref{Section Topological Specification}).

For the rest of this section we prove Theorem  \ref{Theorem 1}. First we demonstrate that the assumption in the statement of Theorem  \ref{Theorem 1} implies that the Large Deviation Principle holds with respect to the topology $\tilde{\mathcal{T}}$. 

\begin{lemma}\label{Lemma tilde T LDP}
The sequence of probability laws $\lbrace \Pi^n \rbrace_{n\in\mathbb{Z}^+}$ satisfy a Large Deviation Principle with respect to the topology $\tilde{\mathcal{T}}$ on $ \mathcal{P}\big(\mathcal{D}\times\mathcal{P}(\mathcal{D})\big)$. That is, for open $O,F \in \mathcal{B}\big(\mathcal{P}(\mathcal{D} \times \mathcal{P}(\mathcal{D}))\big)$, where $O$ is open with respect to $\tilde{\mathcal{T}}$ and $F$ is closed with respect to $\tilde{\mathcal{T}}$.
\begin{align}
\lsup{n}\alpha_n^{-1} \log \Pi^n\big( \hat{\mu}^n_* \in F \big) &\leq -\inf_{\nu \in F}I(\nu)\\
\linf{n}\alpha_n^{-1}\log \Pi^n\big( \hat{\mu}^n_* \in O \big) &\geq -\inf_{\nu \in O}I(\nu),\label{eq: LDP assumed 1 0 0 0} 
\end{align}
where $\alpha_n \to \infty$ as $n\to\infty$. 
\end{lemma}
\begin{proof}
It is proved in Lemma \ref{Lemma Upsilon Dense} that $ \mathcal{P}\big(\mathcal{D}\times\mathcal{P}(\mathcal{D})\big)$ is compact with respect to the topology $\tilde{\mathcal{T}}$. This means that the sequence of probability laws $\lbrace \Pi^n \rbrace_{n\in\mathbb{Z}^+}$ is exponentially tight with respect to the topology $\tilde{\mathcal{T}}$. The Lemma thus follows from the Inverse Contraction Principle \cite[Corollary 4.2.6]{Dembo1998} and the assumption in the statement of Theorem \ref{Theorem 1}.
\end{proof}
The Large Deviations result of Theorem \ref{Theorem Psi} now follows from an application of Varadhan's contraction principle (\cite[Theorem 4.2.1]{Dembo1998}) to Lemma \ref{Lemma tilde T LDP}. We also make use of the continuity of the map $\Psi$ in Lemma \ref{Lemma Psi Limit}.

\bibliographystyle{plain}
\bibliography{neuralbib}
\end{document}